\documentclass[microtype]{gtpart}     
\usepackage[all]{xy}
\usepackage{enumerate}

\title{Categories and orbispaces}

\author{Stefan Schwede}
\address{Mathematisches Institut, Universit{\"a}t Bonn, Germany}
\email{schwede@math.uni-bonn.de}

\keyword{category}
\keyword{orbispace}
\keyword{global homotopy theory}
\subject{primary}{msc2010}{55P91}

\arxivreference{1810.06632}

\numberwithin{equation}{section}
\newtheorem{thm}[equation]{Theorem}    
\newtheorem{prop}[equation]{Proposition}  
\newtheorem{cor}[equation]{Corollary}
\theoremstyle{definition}
\newtheorem{defn}[equation]{Definition}    
\newtheorem{con}[equation]{Construction}
\newtheorem{eg}[equation]{Example}
\newtheorem{rem}[equation]{Remark} 

\renewcommand{\theequation}{\thesection.\arabic{equation}}
\DeclareMathOperator{\colim}{colim}
\DeclareMathOperator{\gl}{gl}
\DeclareMathOperator{\aut}{aut}
\DeclareMathOperator{\map}{map}
\DeclareMathOperator{\Id}{Id}
\DeclareMathOperator{\Ex}{Ex}
\DeclareMathOperator{\Sd}{Sd}
\DeclareMathOperator{\ev}{ev}
\DeclareMathOperator{\Ho}{Ho}

\DeclareMathOperator{\pos}{pos}
\DeclareMathOperator{\Mod}{-mod}
\DeclareMathOperator{\op}{op}
\DeclareMathOperator{\Orb}{Orb}
\DeclareMathOperator{\Fun}{Fun}

\newcommand{\mR}{{\mathbb R}}
\newcommand{\mN}{{\mathbb N}}
\newcommand{\Cc}{{\mathcal C}}
\newcommand{\Dc}{{\mathcal D}}
\newcommand{\Gc}{{\mathcal G}}
\newcommand{\Ic}{{\mathcal I}}
\newcommand{\Lc}{{\mathcal L}}
\newcommand{\Mc}{{\mathcal M}}
\newcommand{\Nc}{{\mathcal N}}

\newcommand{\Uc}{{\mathcal U}}
\newcommand{\Oc}{{\mathcal O}}
\newcommand{\Xc}{{\mathcal X}}
\newcommand{\bL}{{\mathbf L}}
\newcommand{\bO}{{\mathbf O}}
\newcommand{\bT}{{\mathbf T}}
\newcommand{\cat}{{\mathbf {cat}}}
\newcommand{\poset}{{\mathbf {poset}}}
\newcommand{\spc}{spc}
\newcommand{\sset}{\mathbf{sset}}
\newcommand{\sets}{\mathbf{sets}}
\newcommand{\grp}{\mathbf{grp}}
\newcommand{\iso}{\cong}

\newcommand{\xra}{\xrightarrow}
\newcommand{\xla}{\xleftarrow}
\renewcommand{\to}{\longrightarrow}

\begin{document}

\begin{abstract}
Constructing and manipulating homotopy types from categorical
input data has been an important theme in algebraic topology for decades.
Every category gives rise to a `classifying space', 
the geometric realization of the nerve.
Up to weak homotopy equivalence, every space is the classifying space of a small category.
More is true: the entire homotopy theory of topological
spaces and continuous maps can be modeled by categories and functors.
We establish a vast generalization of the equivalence of the homotopy theories
of categories and spaces: small categories
represent refined homotopy types of orbispaces whose underlying coarse moduli space
is the traditional homotopy type hitherto considered.

A {\em global equivalence} is a functor $\Phi:\Cc\to\Dc$ between small categories
with the following property: for every finite group $G$, the functor
$G \Phi:G\Cc\to G\Dc$ induced on categories of $G$-objects is a weak equivalence.
We show that the global equivalences are part of a model structure 
on the category of small categories, which is moreover Quillen
equivalent to the homotopy theory of orbispaces in the sense of Gepner and Henriques.
Every cofibrant category in this global model structure is opposite to a
{\em complex of groups} in the sense of Haefliger.
\end{abstract}

\maketitle


\section*{Introduction}

Constructing and manipulating homotopy types from categorical
input data has been an important theme in algebraic topology for decades.
Every category gives rise to a topological space, sometimes called
its `classifying space', by taking the geometric realization of the nerve.
Up to weak homotopy equivalence, every space is the classifying space of a category.
But much more is true: the entire {\em homotopy theory} of topological
spaces and continuous maps 
can be modeled by categories and functors.

In this paper we prove a substantial strengthening of the equivalence of homotopy theories
of categories and spaces. Our perspective is that categories 
don't just represent mere homotopy types; rather, small categories
represent refined homotopy types of orbispaces whose underlying `coarse moduli space'
is the traditional non-equivariant homotopy type hitherto considered.
We emphasize that we talk about plain, traditional 1-categories here,
with no additional structure, enhancement or higher categorical bells and whistles attached.

One can informally think of an orbispace as the quotient of a space by the
action of a finite group, but with information about the isotropy groups of the
action built into the formalism.
In this sense, orbispaces are to topological spaces as orbifolds are to manifolds.
The homotopy theory of orbispaces arises in several different contexts.
The orbispaces we use here were introduced by Gepner and Henriques 
in \cite{gepner-henriques}, who identify them with a specific kind of topological stacks.
As we explain in \cite{schwede:orbispace},
orbispaces can be interpreted as `spaces with an action of the universal compact Lie group',
and they also model unstable global homotopy theory in the sense of \cite{schwede:global}.
We refer to Remark \ref{rk:other orbispc} below for more information, background and
motivation about orbispaces.

By definition, a functor $\Phi:\Cc\to\Dc$ between small categories
is a {\em global equivalence} if for every finite group $G$
the functor $G \Phi:G\Cc\to G\Dc$ induced on categories of $G$-objects
is a weak equivalence. In other words, to qualify as a global equivalence,
the functor $\Phi$ must induce weak homotopy equivalences of simplicial sets
\[ N (G \Phi)\ : \ N(G\Cc)\ \to \ N(G\Dc) \]
for all finite groups $G$, where $N$ denotes the nerve construction.
We show in Theorem \ref{thm:global cat} that the global equivalences
are the weak equivalences in a certain model structure 
on the category of small categories, the {\em global model structure}.

The connection between orbispaces and the global homotopy theory of categories
comes from the natural structure that relates
the categories $G\Cc$ and $K\Cc$ for different groups $G$ and $K$.
The assignment $G\mapsto G\Cc$ has a specific 2-functoriality: every group homomorphism
$\alpha:K\to G$ gives rise to a restriction functor $\alpha^*:G\Cc\to K\Cc$,
and every group element $g\in G$ gives rise to a
natural isomorphism $l_g:\alpha^*\Longrightarrow (c_g\circ\alpha)^*$,
where $c_g:G\to G$ is the inner automorphism $c_g(\gamma)=g\gamma g^{-1}$.
This data altogether forms a strict contravariant 2-functor 
from the 2-category $\grp$ of groups, homomorphism and conjugations, 
see Construction \ref{def:Grp}.
Taking nerves turns the 2-category $\grp$ into a category $\bO_{\gl}=N(\grp)$ 
enriched in simplicial sets,
and the assignment $G\mapsto N(G \Cc)$ becomes a contravariant simplicial functor
from $\bO_{\gl}$ to simplicial sets.
The geometric realization of the simplicial category $\bO_{\gl}$
is the global orbit category as defined by Gepner
and Henriques \cite{gepner-henriques}, so such a functor is precisely an {\em orbispace}.
We call the orbispace with values $N(G\Cc)$ the {\em global nerve} of the category $\Cc$.
Theorem \ref{thm:cat and orbispace} says that the global nerve functor
induces an equivalence from the global homotopy theory of small categories
to the homotopy theory of orbispaces.
In particular, every orbispace is equivalent to the global nerve of some small category.

The global homotopy theory of categories also has a close connection to
{\em complexes of groups} in the sense of Haefliger \cite{haefliger}.
Every non-equivariant homotopy type can be modeled by a poset,
and poset categories already model the full homotopy theory
of spaces. Indeed, every cofibrant category in Thomason's model structure
\cite{thomason:model cat cat} is a poset, and Raptis showed \cite{raptis} that the
inclusion of posets into small categories is a Quillen equivalence.
The {\em global} homotopy type of a category, however, makes essential use of
automorphisms, and only very special global homotopy types 
arise as global nerves of poset
(the constant ones, compare Example \ref{eg:posets}).  
The role of poset categories in our context is played by {\em complexes of groups}
in the sense of \cite{bridson-haefliger, haefliger},
or rather by the variant that allows for non-injective transition homomorphisms,
see Definition \ref{def:complex of groups}.
We show in Theorem \ref{thm:cofibrant is complex}
that the opposite of every globally cofibrant category is associated to a complex of groups.

The original motivation for this project was the connection between
categories and orbi\-spaces explained in the previous paragraphs.
However, the refined model structure and the Quillen equivalence to a category
of simplicial functors work substantially more generally. 
Instead of looking at $G$-objects in $\Cc$,
we can consider nerves of functor categories $\Fun(I,\Cc)$ from specific
small categories $I$. The only serious hypothesis we impose on the categories
$I$ is that they are {\em strongly connected}, i.e., there is at least one morphism
between every ordered pair of objects. This condition is crucial to gain homotopical
control over certain pushouts of categories, 
see Theorem \ref{thm:Dwyer M pushout} for the precise statement.
We develop the theory in Sections \ref{sec:global cats}
and \ref{sec:cats vs sfunctors} in this generality,
and then specialize to the motivating case of equivariant objects 
in Section \ref{sec:group case}.

{\bf Acknowledgments} 
This paper was written in the spring of 2018, while I was a long term guest at the
Centre for Symmetry and Deformation at the University of Copenhagen;
I would like to thank the Center for Symmetry and Deformation 
for the hospitality and the stimulating atmosphere during this visit.
I would particularly like to thank Jesper Grodal for several interesting
conversation on the topics of this paper. 
I am grateful to Tobias Lenz and Viktoriya Ozornova for
providing the argument that shows left properness of the $\Mc$-model structure,
and to Markus Hausmann for various interesting comments and suggestions.
The author is a member of the Hausdorff Center for Mathematics
at the University of Bonn.

\section{A refined model structure for categories}
\label{sec:global cats}

In this section we introduce a refined notion of equivalences for functors between
small categories, via functor categories out of specific kinds of categories.
In Theorem~\ref{thm:M cat} we extend these equivalences to the `$\Mc$-model structure'
on the category of small categories.
In Theorem~\ref{thm:Mcat and Morbispace}, we then show that the $\Mc$-model structure
is Quillen equivalent to a certain category of simplicial presheaves.
In Section \ref{sec:group case} we specialize to the main case of interest,
namely when $\Mc$ is the class of classifying categories of finite groups.

\bigskip

We begin with some history and background, to put our results into context.
Along the way, we also fix our notation.
The insight that categories and functors model the homotopy theory of spaces
via the geometric realization of the nerve construction developed in stages.
Geometric realization and singular complex provide homotopically meaningful
ways to pass back and forth between topological spaces and simplicial sets,
so our summary focuses on the relationship between categories and simplicial sets.
I do not know the precise origin of the nerve construction for categories,
but the name seems to derive from the earlier idea of the nerve of a covering.
The earliest published account I know of that explicitly uses the nerve of a category
is the book of Gabriel and Zisman \cite[II.4]{gabriel-zisman}.
The first reference to the term `nerve' is Segal's \cite{segal:classifying spaces}
who acknowledges that `(...) all the ideas are implicit in the work of Grothendieck'.

We write $\cat$ for the category of small categories and functors.
We let $\Delta$ denote the simplicial indexing category with objects
the finite totally ordered sets $[n]=\{0\leq 1\leq \dots\leq n\}$, for $n\geq 0$,
and morphisms the weakly monotone maps.
We write $\sset$ for the category of simplicial sets, i.e., contravariant
functors from $\Delta$ to sets.
A fully faithful embedding $p:\Delta\to\cat$
sends $[n]$ to its poset category, i.e., the category with object set $\{0,1,\dots,n\}$
and a unique morphism from $i$ to $j$ whenever $i\leq j$.
A weakly monotone map then extends uniquely to a functor between the associated
poset categories. 
The {\em nerve} of a small category $\Cc$ is the composite simplicial set
\[ \Delta^{\op} \ \xra{\ p^{\op}\ } \ \cat^{\op} \ \xra{\ \cat(-,\Cc)}\ \sets \ .\]
So $(N\Cc)_n$ is the set of functors $[n]\to \Cc$, 
and  such a functor is uniquely determined
by a string of $n$ composable morphisms, namely the images of the morphisms $i-1\leq i$
for $i=1,\dots,n$. For varying $\Cc$, 
the nerve construction becomes a functor $N:\cat\to\sset$.

A functor $F:\Cc\to \Dc$ is a {\em weak equivalence} if the induced morphism
of simplicial sets  $N F:N \Cc\to N\Dc$ is a weak equivalence, i.e., becomes a homotopy
equivalence after geometric realization.
The nerve functor is fully faithful, and its essential image
consists of the 2-coskeletal simplicial sets, see \cite[Prop.\,2.2.3]{illusie-II}.
The nerve functor has a left adjoint $c:\sset\to\cat$, see \cite[II.4]{gabriel-zisman}.
The category $c(X)$ associated with a simplicial set $X$ has 
as objects the vertices of $X$. The morphisms in $c(X)$
are freely generated by the 1-simplices of $X$, with $f\in X_1$
considered as a morphism from $X(d_1)(f)$ to $X(d_0)(f)$, modulo the relations
$X(s_0)(x)=\Id_x$ for all $x\in X_0$, and
\[ X(d_1)(z)\ = \ X(d_0)(z)\circ X(d_2)(z) \]
for every $z\in X_2$. The functor $c$ is not only left adjoint, 
but also left inverse to the nerve; in other words, for every small category $\Cc$,
the adjunction counit
\[ c(N\Cc)\ \to \ \Cc  \]
is an isomorphism of categories.
Unfortunately, the `categorification functor' $c$ 
has poor homotopical properties in general.

The first known fully homotopical functor from simplicial sets to categories 
that is homotopy-inverse to the nerve
seems to be the simplex category construction, see \cite[Thm\,3.2]{illusie-II};
this implies in particular that the nerve functor induces 
an equivalence of homotopy categories.
According to Illusie \cite[Sec.\,2.3]{illusie-II}, these facts are due to Quillen.
There are other constructions that turn simplicial sets 
into categories in a homotopically meaningful way and that are homotopy inverse
to the nerve, and we refer to the introduction of \cite{fritsch-latch} 
for further references.

A strengthening and a particularly structured way of comparing the homotopy theories of
categories and simplicial sets was obtained by Thomason \cite{thomason:model cat cat}.
He realized that the composite
\[ c\circ\Sd^2 \ : \ \sset \ \to \ \cat \]
of the 2-fold iterated subdivision in the sense of Kan \cite{kan:on css}
and the categorification functor
can also serve as a homotopy-inverse to the nerve functor,
as both natural transformations
\[ K \ \xla{\qquad}\ \Sd^2 K \ \xra{\ \eta_{\Sd^2 K}}\ N(c(\Sd^2 K)) \]
are weak equivalences of simplicial sets; a proof can be found in \cite[4.12 (v)]{fritsch-latch}.
A special feature of the functor $c\circ \Sd^2$
is that it takes finite simplicial sets (i.e., those with finitely many non-degenerate
simplices) to finite categories (i.e., those with finitely many objects and morphisms).
In \cite{thomason:model cat cat}, Thomason extended the weak equivalences to
a model structure on the category of small categories, designed so that
the adjoint functor pair $(c\circ\Sd^2,\Ex^2\circ N)$ becomes
a Quillen equivalence to Quillen's model structure \cite[II.3]{Q} on simplicial sets.

Thomason's model structure and Quillen equivalence 
have been generalized to an equivariant context
by Bohmann, Mazur, Osorno, Ozornova, Ponto and Yarnall \cite{BMOOPY}.
For a group $G$, they construct 
a model structure on the category of small $G$-categories
and a Quillen equivalence to the category of $G$-spaces.
Our results here go in a different direction:
we start from a category with no extra structure,
and our notion of `global equivalences'
uses $G$-objects in $\Cc$ (as opposed to $G$-actions on $\Cc$).

\medskip

After these historical remarks, we move on towards the new mathematics.
We will refine the weak equivalences of categories by testing through functors
from certain indexing categories. An important special case is to test
by the classifying categories for all finite groups,
in which case the resulting homotopy theory will turn out to be equivalent
to the orbispaces in the sense of Gepner and Henriques \cite{gepner-henriques};
we discuss this special case in more detail in Section \ref{sec:group case}.

Already in Thomason's paper \cite{thomason:model cat cat}
a certain class of functors, the {\em Dwyer maps}, plays a special role.
This class of functors will also be important for us.
We recall from \cite[Def.\,4.1]{thomason:model cat cat} 
that a functor $i:A\to B$ between small categories 
is a {\em Dwyer map} if it is 
\begin{itemize}
\item fully faithful and injective on objects;
\item a {\em sieve}, i.e., for every $B$-morphism $f:b\to a$ with $a$ in $A$,
  both the object $b$ and the morphism $f$ belong to $A$; and
\item there is a full subcategory $W$ of $B$ that is a cosieve in $B$, contains $A$ and
  such that the inclusion $A\to W$ admits a right adjoint.
\end{itemize}
I have been unable to find out why Thomason attaches Bill Dwyer's name to these
functors.
Dwyer maps can be informally thought of as categorical analogs of the inclusion
of a neighborhood deformation retract; they play a similar role in the context of small 
categories as the one played by h-cofibrations in the context of topological spaces.
Dwyer maps are stable under cobase change
\cite[Prop.\,4.3]{thomason:model cat cat},
composition (finite and sequential) \cite[Lemma 5.3]{thomason:model cat cat},
and cobase change along Dwyer maps are homotopical for
weak equivalences \cite[Cor.\,4.4]{thomason:model cat cat}. 

Cisinski noted \cite{cisinski-dwyer}
that contrary to Thomason's claims 
in \cite[Lemma 5.3 (3), Prop.\,5.4]{thomason:model cat cat},
Dwyer maps are {\em not} closed under retracts,
and there are cofibrations in Thomason's model structure that are not Dwyer maps.
The missing retract property of Dwyer maps is not used in the proof of
the model category axioms nor the Quillen equivalence;
Thomason does use it in the proof of properness, but Cisinski fixes this gap
in \cite{cisinski-dwyer}.

We need additional properties of Dwyer maps
that are not explicitly stated in \cite{thomason:model cat cat}.
The first two properties in the following proposition are straightforward;
the third one in Theorem~\ref{thm:Dwyer M pushout} is more involved.
Given two small categories $I$ and $A$, we denote by $\Fun(I,A)$
the category of functors from $I$ to $A$, with natural transformations as
morphisms.

\begin{prop}\label{prop:preserve Dwyer}
For every Dwyer map $i:A\to B$ and every small category $I$,
the functor $i\times I:A\times I\to B\times I$ and the
functor $\Fun(I,i):\Fun(I,A)\to \Fun(I,B)$ is a Dwyer map.
\end{prop}
\begin{proof}
  We may assume that $i$ is the inclusion of a full subcategory $A$ of $B$ that
  is a sieve in $B$; moreover, there is a cosieve $W$ of $B$ that contains $A$ and
such that the inclusion $A\to W$ has a right adjoint.

The category $A\times I$ is a full subcategory and a sieve in $B\times I$,
$W\times I$ is a full subcategory and a cosieve in $B\times I$,
and the inclusion $A\times I\to W\times I$ has a right adjoint.
Similarly, the category $\Fun(I, A)$ is a full subcategory and a sieve in $\Fun(I, B)$,
$\Fun(I, W)$ is a full subcategory and a cosieve in $\Fun(I, B)$,
and the inclusion $\Fun(I, A)\to\Fun(I,W)$ has a right adjoint.
\end{proof}

At various places we need homotopical control over certain pushouts of categories.
In general, pushouts of categories can be a mess (i.e., hard to identify explicitly),
and the homotopy type of the nerve of a pushout does not bear any apparent relationship
to the nerves of the input data.
However, Thomason showed in \cite[Prop.\,4.3]{thomason:model cat cat}
that pushouts along Dwyer maps behave well homotopically:
the nerve functor takes such pushouts to homotopy pushouts of simplicial sets.

Since our refined equivalences depend on categories of functors,
we also need control about functors from a fixed category
into certain pushouts of categories.
The functor $\Fun(I,-):\cat\to\cat$ does not in general preserve pushouts.
However, Theorem \ref{thm:Dwyer M pushout} below 
shows that if $I$ is strongly connected, then $\Fun(I,-)$
preserves pushouts along Dwyer maps.

\begin{con}[Pushouts along Dwyer maps]\label{con:Dwyer pushout}
An explicit description of a pushout along a fully faithful inclusion
of categories $i:A\to B$ is given in \cite[Prop.\,5.2]{fritsch-latch}. 
If the inclusion is a Dwyer map, then the description simplifies,
see for example the proof of \cite[Lemma 2.5]{BMOOPY}.
We recall this description in some detail here.

We let $i:A\to B$ be a Dwyer map and $k:A\to C$ any functor.  
We let $Z$ be the cosieve generated by $A$ in $B$, i.e., the full subcategory
of all $B$-objects that admit a $B$-morphism from an object in $A$.
Then the inclusion $A\to Z$ has a right adjoint by the hypothesis that $i$ is a Dwyer map.
Since $A$ is a full subcategory of $Z$,
we can choose a right adjoint $r:Z\to A$ that is the identity on $A$
and such that the adjunction counit $\epsilon:i r\to\Id_Z$ is the identity on $A$.
We define a category $D$ that will turn out to be a pushout of $i:A\to B$ and $k:A\to C$.
We let $V$ be the full subcategory of $B$ whose objects are the ones 
that do not belong to $A$.
The objects of $D$ are the disjoint union of the objects of $C$ and
the objects of $V$.
The morphism sets in $D$ are defined as
\[ D(d,d')\ = \
\begin{cases}
\ C(d,d')  & \text{ if $d,d'\in C$,}\\ 
C(d,k(r(d')))  & \text{ if $d\in C$ and $d'\in V\cap Z$, and}\\ 
\ B(d,d')  & \text{ if $d,d'\in V$.}
\end{cases}
\]
Moreover, there are no $D$-morphisms from objects in $V$ to objects in $C$,
and there are no $D$-morphisms from objects in $C$ to objects in $V\backslash Z$.
Composition is defined by requiring that $C$ and $V$ become full subcategories of $D$;
moreover, for $c,c'\in C$ and $z,z'\in V\cap Z$, composition
  \[ D(c,z)\times D(c',c) \ \to \  D(c',z)  \]
    is composition in $C$, and composition
    \[ D(z,z')\times D(c,z)  \ \to \  D(c,z')  \]
    is defined by
    \[ \beta \circ \gamma\ = \ k(r(\beta))\circ \gamma  \ . \]
It is straightforward to check that composition in $D$ is well-defined, 
associative and unital, so we have indeed defined a category.

Now we introduce the functors that express the category $D$
as a pushout of $i:A\to B$ and $k:A\to C$.
The functor $j:C\to D$ is simply the inclusion.
We define a functor $h:B\to D$ on objects by
\[ h(b)\ = \
\begin{cases}
  k(b) & \text{ if $b\in A$, and}\\
  \ b & \text{ if $b\in V$.}
\end{cases} \]
On morphisms, $h$ is given by
\[ h(f:b\to b')\ = \
\begin{cases}
     k(r(f))  & \text{ if $b\in A$ and $b'\in Z$, and}\\
  \quad f & \text{ if $b,b'\in V$.}
\end{cases} \]
Since $A$ is a sieve in $B$, there are no morphisms from
objects in $V$ to objects in $A$;
similarly, since $A\subset Z$ and $Z$ is a cosieve in $B$, 
there are no morphisms from objects in $A$ to objects in $B\backslash Z$.
So the above recipe is a complete definition of $h$ on morphisms.
The compatibility of this assignment with identity morphisms and composition
is straightforward from the definitions, so $h:B\to D$ is a functor.
The condition $h\circ i=j\circ k$ as functors $A\to D$ was built into
the definitions.
For later use we observe a relation that comes from the fact that $\epsilon:i r\to\Id_Z$
is the counit of an adjunction.
Given $c\in C$ and $z\in V\cap Z$, we have
\[ D(c,z)\ = \ C(c,k(r(z))) \ = \ D(c,k(r(z)))\ .\]
So a $C$-morphism $f:c\to k(r(z))$
is both a $D$-morphism $j(c)\to h(z)$
and a $D$-morphism $j(c)\to j(k(r(z)))=h(i(r(z)))$.
Moreover, with these two different interpretations,
the factorization 
\begin{equation}\label{eq:can fact}
  f \ = \   h(\epsilon_z) \circ j(f)  \ : \  j(c)\ \to \ h(z)
\end{equation}
holds in $D$.

The final step is to verify that the functors $h:B\to D$ and $j:C\to D$
enjoy the universal property of a pushout of $i:A\to B$ and $k:A\to C$. 
So we let $\varphi:B\to E$ and $\psi:C\to E$
be functors such that $\varphi i=\psi k$. We must show that there 
is a unique functor $\kappa:D\to E$ satisfying 
\[\kappa h\ =\ \varphi \text{\quad and\quad}  \kappa j\ =\ \psi\ . \]
Since the objects of $D$ are the disjoint union of the objects of $C$ and $V$,
the behavior of the functor $\kappa$ on objects is forced to be
\[ \kappa(d)\ = \
\begin{cases}
  \psi(c) & \text{ if $d=j(c)$ for $c\in C$, and}\\
  \varphi(v) & \text{ if $d=h(v)$ for $v\in V$.}
\end{cases} \]
The behavior on morphisms is similarly forced:
\[
\kappa(f)\ = \
\begin{cases}
\quad \psi(g) & \text{\ if $f=j(g)$ for a $C$-morphism $g$,}\\
\quad \varphi(g) &\text{\ if $f=h(g)$ for a $V$-morphism $g$, and}\\
\varphi(\epsilon_z)\circ\psi(f) & \text{\ if $f\in D(c,z)=C(c,k(r(z)))$ for $c\in C$ and $z\in V\cap Z$.}
\end{cases}
\]
The third case uses the relation \eqref{eq:can fact}.
This proves that there is at most one functor $\kappa$ that satisfies 
$\kappa h= \varphi$ and $\kappa j=\psi$.
Conversely, given $\varphi$ and $\psi$, we can define $\kappa$ by the equations above.
We omit the verification that then $\kappa$ is really a functor, i.e.,
compatible with composition. This completes the verification that
the functors $h:B\to D$ and $j:C\to D$ make the category $D$ into a pushout of the functors
$i:A\to B$ and $k:A\to C$. 
\end{con}

\begin{defn}
A category is {\em strongly connected}
if for every pair of objects $x,y$ there exists a morphism $x\to y$.  
\end{defn}

In particular, the nerve of a strongly connected category is connected,
but not conversely. For example, the poset category of $\{0\leq 1\}$
is connected, but not strongly connected.
Particular examples of strongly connected categories
are categories with a single object, i.e., classifying categories
for monoids. 

Now we come to a key technical result.

\begin{thm}\label{thm:Dwyer M pushout}
Consider a pushout square of small categories on the left
 \[ \xymatrix{ A \ar[r]^-k \ar[d]_i & C \ar[d]^j &&
\Fun(I, A) \ar[rr]^-{\Fun(I, k)} \ar[d]_{\Fun(I, i)} && \Fun(I, C) \ar[d]^{\Fun(I, j)}\\
   B \ar[r]_-h & D &&
   \Fun(I, B) \ar[rr]_-{\Fun(I, h)} && \Fun(I, D)
} \]
such that $i$ is a Dwyer map.
Then for every strongly connected category $I$, the commutative diagram
of categories on the right is a pushout
and the canonical morphism
\[ N\Fun(I,h)\cup N\Fun(I,j)\ : \ N\Fun(I,B)\cup_{N \Fun(I,A)} N\Fun(I,C) \ \to \ 
N\Fun(I, D)\]
is a weak equivalence of simplicial sets.
\end{thm}
\begin{proof}
We exploit the explicit description of the pushout along a Dwyer map
that we recalled in Construction \ref{con:Dwyer pushout};
in other words, we assume that the category $D$
and the functors $h:B\to D$ and $j:C\to D$ are the ones defined there.
In particular, $D$ contains $C$ and $V=B\backslash A$ 
as disjoint full subcategories of $D$ that together contain all objects.

We let $I$ be a strongly connected category and $F:I\to D$
a functor. Since there is a morphism $x\to y$ for every pair
of objects in $I$, but there are no $D$-morphisms from objects in $V$ to objects in $C$,
the functor must take values entirely in $C$ or entirely in $V$.
Since $C$ and $V$ are full subcategories of $D$, the functor
$F$ lifts uniquely to a functor to $C$,
or it lifts uniquely to a functor to $V$.
In other words, the objects of the functor category $\Fun(I,D)$ are the disjoint union 
of the objects of $\Fun(I,C)$ and the objects of $\Fun(I,V)$.
Also, $\Fun(I,V)=\Fun(I,B)\backslash \Fun(I,A)$.

Now we identify natural transformations between two functors $F,G:I\to D$. 
If both functors take values in $C$, any natural transformation arises 
from unique $C$-valued natural transformation, because $C$ is a full subcategory.
Similarly, if both functors take values in $V$, all $D$-valued natural
transformations arise uniquely from $V$-valued natural transformations.

It remains to identify natural transformations $\tau:F\Longrightarrow G$
in the `mixed' case.
There are no $D$-morphisms from objects in $V$ to objects in $C$, hence also
no natural transformations from $V$-valued functors to $C$-valued functors.
In the remaining case we consider functors  $F:I\to C$ and $G:I\to V$.
As before we let $Z$ denote the cosieve generated by $A$ in $B$.
Since there are no $D$-morphisms from objects in $C$ to objects in $V\backslash Z$, 
the existence of $\tau$ implies that $G$ must takes values in $V\cap Z$.
The data of the transformation $\tau$ consists of $D$-morphisms 
\[ \tau(i)\ \in\  D(F(i),G(i))\ =\ C(F(i),k(r(G(i)))) \]
for all objects $i$ of $I$.
The retraction functor $r:Z\to A$ and the functor $k:A\to C$
give rise to a composite functor
\[ k\circ r\circ G \ : \ I \ \to \ C\ . \]
We can thus interpret the morphisms $\{\tau(i)\}_{i\in I}$ 
as the candidate components for a natural transformation of
$C$-valued functors $\tau':F\Longrightarrow k\circ r\circ G$.
Naturality of $\tau$ means the relation
\[ \tau(j)\circ F(\alpha)\ = \ G(\alpha)\circ \tau(i) 
\text{\quad in\quad}  D(F(i),G(j))\ =\ C(F(i),k(r(G(j)))) \]
for all $I$-morphisms $\alpha:i\to j$.
Composition in the category $D$ was defined so that this relation in fact means
the relation
\[ \tau(j)\circ F(\alpha)\ = \ k(r(G(\alpha)))\circ \tau(i) \ : \ F(i)\ \to \ k(r(G(j)))\]
among $C$-morphisms.
This shows that the collection of morphisms $\{\tau(i)\}_{i\in I}$
forms a natural transformation of $D$-valued functors $F\Longrightarrow G$ 
if and only if it
forms a natural transformation of $C$-valued functors $F\Longrightarrow k\circ r\circ G$,
i.e., 
\[ \Fun(I,D)(F,G)\ = \ \Fun(I,C)(F,k\circ r\circ G) \ .\]
On the other hand, the functor $\Fun(I,i):\Fun(I,A)\to \Fun(I,B)$ is also a Dwyer map
by Proposition \ref{prop:preserve Dwyer}.
So Construction \ref{con:Dwyer pushout} can also be used to describe a
pushout of the functors $\Fun(I,i):\Fun(I, A)\to \Fun(I, B)$ 
and $\Fun(I, k):\Fun(I, A)\to \Fun(I, C)$.
In this description, the objects and morphisms are precisely the
objects and morphisms of $\Fun(I, D)$ as identified above, and composition
also works out in the same way. So the category $\Fun(I, D)$
is indeed a pushout of the functors $\Fun(I, i):\Fun(I, A)\to \Fun(I, B)$ 
and $\Fun(I, k):\Fun(I, A)\to \Fun(I, C)$.

For that last claim we observe that the canonical morphism factors as the composite
\begin{align*}
  N\Fun(I,B)&\cup_{N \Fun(I,A)} N\Fun(I,C) \ \to \\ 
 &N\left( \Fun(I,B)\cup_{\Fun(I,A)} \Fun(I,C)\right) \ \to \ N\Fun(I, D)  \ .
\end{align*}
Since $\Fun(I,i):\Fun(I,A)\to\Fun(I,B)$ is again a Dwyer map 
by Proposition \ref{prop:preserve Dwyer}, the first morphism is
a weak equivalence by Thomason's result \cite[Prop.\,4.3]{thomason:model cat cat}.
The second morphism is an isomorphism by the above.
\end{proof}

\begin{eg}
We give a simple example to illustrate that Theorem \ref{thm:Dwyer M pushout}
need not hold anymore if the indexing category $I$ is connected but not
strongly connected. The example involves the poset categories $p[n]$
associated with the totally ordered sets $[n]=\{0<1<\dots<n\}$
for $n=0,1,2$. We write $d_i:p[n-1]\to p[n]$ for the unique functor that 
is injective on objects and whose image does not contain the object $i$.

The functor $d_1:p[0]\to p[1]$ that embeds the terminal category
as the source of the non-identity morphism of $p[1]$
is a Dwyer map (take $W=p[1]$), and also a cofibration in Thomason's model structure
by \cite[Lemme 1]{cisinski-dwyer}.
The following square is a pushout of small categories:
\[ \xymatrix@C=12mm{ 
p[0]\ar[r]^-{d_0}\ar[d]_{d_1} & p[1]\ar[d]^{d_2} \\
p[1] \ar[r]_-{d_0} & p[2]} \]
The nerve functor does {\em not} take this square to a pushout 
of simplicial set: the canonical morphism
\[ \Delta[1]\cup_{\Delta[0]}\Delta[1]\ = \ 
N(p[1])\cup_{N(p[0])} N(p[1]) \ \to \ N(p[2]) \ = \ \Delta[2]\]
is the inclusion of the inner horn into $\Delta[2]$.
This canonical map is a weak equivalence, however, as predicted by
\cite[Prop.\,4.3]{thomason:model cat cat}.

The category $I=p[1]$ is connected, but not strongly connected.
Moreover, the functor $d_1:I=p[1]\to p[2]$ does not factor through
$d_0$ nor $d_2$. So applying $\Fun(I,-)$ to the above square
does {\em not} yield a pushout of categories. 
\end{eg}

Now we introduce our refined notion of equivalence between small categories.
This depends on a class of finite, strongly connected categories.
The restriction to {\em finite} categories 
(i.e., with finitely many objects and finitely many morphisms)
is not essential and can be replaced by
a global bound on the cardinality, see Remark \ref{rk:infinite}.
The strong connectedness of the test categories is essential, however,
because Theorem \ref{thm:Dwyer M pushout} has no analog for general source categories.

\begin{defn}
Let $\Mc$ be a class of finite, strongly connected categories.
A functor $F:X\to Y$ between small categories is an {\em $\Mc$-equivalence}
if the functor $\Fun(I, F):\Fun(I, X)\to\Fun(I ,Y)$ is a weak equivalence of categories
for every category $I$ in $\Mc$. 
\end{defn}

If $\ast$ is a terminal category with one object and its identity, then
$\Fun(\ast,X)$ is isomorphic to $X$. 
So if $\Mc=\{\ast\}$ consists only of a single terminal category, 
$\Mc$-equivalences are precisely the weak equivalences.

\begin{eg} 
  We call a functor $F:X\to Y$ between small categories a {\em homotopy equivalence}
  if there exists a functor $F':Y\to X$ such that both composites $F\circ F'$ and
  $F'\circ F$ can be related to the respective identity functors by a finite chain of
  natural transformations. For example, $F$ is a homotopy equivalence if it
  is an equivalence of categories, or if it has a left adjoint or a right adjoint.
  
  If $F:X\to Y$ is a homotopy equivalence, then so is
  $\Fun(I, F):\Fun(I, X)\to\Fun(I ,Y)$ for every small category $I$.
  So every homotopy equivalence is an $\Mc$-equivalence.
\end{eg}

\begin{eg}
  The set $P=\mN\times\{0,1\}$ becomes a poset if we declare 
  \[ (i,0)\leq (i,1)\text{\quad and\quad} (i,0)\leq (i+1,1) \]
  for all $i\in \mN$, and every element is in relation with itself.
  The nerve of $P$ is the 1-dimensional, weakly contractible simplicial set 
  that can be pictured as follows:
  
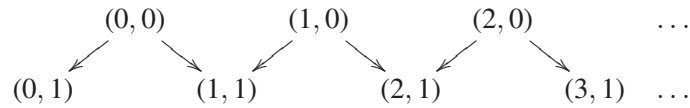
\[ \xymatrix@C=2mm@R=4mm{  &(0,0)\ar[dl]\ar[dr] && (1,0)\ar[dl]\ar[dr] && (2,0)\ar[dl]\ar[dr] &&\dots\\
      (0,1) && (1,1)&& (2,1) && (3,1)&\dots
    } \]
  (continuing indefinitely to the right).
  Because $P$ is a poset, every functor $I\to P$ from a non-empty strongly connected category
  is constant, and the functor $P\to \Fun(I,P)$ that sends an object of $P$
  to the associated constant functor is an isomorphism of categories.
  So the simplicial set $N \Fun(I,P)$ is weakly contractible.
  This shows that the unique functor $P\to \ast$ is an $\Mc$-equivalence
  for every class $\Mc$ of finite, strongly connected categories. 
  However, the functor $P\to\ast$ is {\em not} a homotopy equivalence because
  the identity functor of $P$ cannot be related by a finite chain of natural transformations
  to a constant functor.
\end{eg}

The class of $\Mc$-equivalences of categories is closed under various constructions:

\begin{prop}\label{prop:global equiv basics} 
  Let $\Mc$ be a class of finite, strongly connected categories.
  The class of $\Mc$-equivalences of small categories enjoys the following properties.
  \begin{enumerate}[\em (i)]
  \item 
    A coproduct of $\Mc$-equivalences is an $\Mc$-equivalence.
  \item 
    A finite product of $\Mc$-equivalences is an $\Mc$-equivalence. 
  \item Let $f_n:Y_n\to Y_{n+1}$ be an $\Mc$-equivalence between small categories, 
    for $n\geq 0$. 
    Then the canonical functor $f_\infty:Y_0\to Y_\infty$ to the colimit 
    of the sequence $\{f_n\}_{n\geq 0}$ is an $\Mc$-equivalence.
  \item  Let 
    \[ \xymatrix{ A \ar[r]^-k \ar[d]_i & C \ar[d]^j\\
      B \ar[r]_-h & D } \]
    be a pushout square of small categories such that $k$ is an $\Mc$-equivalence.
    If in addition $i$ or $k$ is a Dwyer map,
    then the functor $h$ is an $\Mc$-equivalence.
  \end{enumerate}
\end{prop}
\begin{proof}
  Part~(i) holds because $\Fun(I,-)$, for strongly connected categories $I$, 
  and taking nerves both commute with disjoint unions,
  and disjoint unions of weak equivalences of simplicial sets are weak equivalences.
  Part~(ii) holds because $\Fun(I,-)$ and nerves commute with products
  and a finite product of weak equivalences of simplicial sets is a weak equivalence.
  
  (iii) For finite categories $I$, the functor $\Fun(I,-):\cat\to\cat$
  commutes with filtered colimits of categories. 
  The nerve functor $N:\cat\to\sset$ also
  commutes with filtered colimits. So the canonical morphism
  \[ \colim_{n\geq 0} N\Fun(I, Y_n)\ \to \ N\Fun(I,Y_\infty) \]
  is an isomorphism of simplicial sets.
  Part (iii) is then a consequence of the fact that filtered colimits of simplicial
  sets are fully homotopical for weak equivalences, see for example
  \cite[Thm.\,4.6]{barnea-schlank} or \cite[Cor.\,5.1]{raptis-rosicky}.

  (iv) We let $I$ be a strongly connected category from $\Mc$.
  If $i$ or $k$ is a Dwyer map, then so is $\Fun(I,i):\Fun(I,A)\to\Fun(I,B)$ 
  or $\Fun(I,k):\Fun(I,A)\to\Fun(I,C)$,
  respectively. Dwyer maps are in particular injective on objects and faithful,
  so the morphism $N\Fun(I,i):N\Fun(I,A)\to N\Fun(I,B)$ 
  or the morphism $N\Fun(I,k):N\Fun(I,A)\to N\Fun(I,C)$
  is injective. Since $N\Fun(I,k):N\Fun(I,A)\to N\Fun(I,C)$ is a weak equivalence, 
  the canonical map
  \[ N\Fun(I,B) \ \to \ N\Fun(I,B)\cup_{N\Fun(I,A)} N\Fun(I,C)\]
  is a weak equivalence of simplicial sets. 
  The morphism $N\Fun(I,h):N\Fun(I,B)\to N\Fun(I,D)$
  factors as this latter weak equivalence followed by the weak equivalence of
  Theorem \ref{thm:Dwyer M pushout}.
  Hence $N\Fun(I,h)$ is a weak equivalence of simplicial sets for every $I$ in $\Mc$,
  and so $h$ is an $\Mc$-equivalence of categories.
\end{proof}

We turn to the construction of the $\Mc$-model structure 
on the category of small categories, see Theorem \ref{thm:M cat}.
We will make essential use of Thomason's `non-equivariant'
model structure \cite{thomason:model cat cat} for $\cat$.
The weak equivalences in the Thomason model structure are the weak equivalence of categories,
i.e., functors that induce a weak equivalence of simplicial sets on nerves.
A functor $f:X\to Y$ is a fibration in Thomason's model structure
if and only if the morphism $\Ex^2(N f): \Ex^2(N X)\to \Ex^2(N Y)$
is a Kan fibration of simplicial sets, where $\Ex$ is right adjoint to Kan's subdivision functor
\cite{kan:on css}, and $\Ex^2$ is the twofold iterate of $\Ex$.
The Thomason cofibrations are defined by the left lifting property with respect
to acyclic fibrations.

\begin{defn}
  Let $\Mc$ be a class of finite, strongly connected categories.
  A functor between small categories $f:X\to Y$ is an {\em $\Mc$-fibration}
  if the functor $\Fun(I,f):\Fun(I,X)\to\Fun(I,Y)$
  is a fibration in the Thomason model structure for every category $I$ in $\Mc$.
  A functor is an {\em $\Mc$-cofibration} it if has the left lifting
  property with respect to all functors that are simultaneously
  $\Mc$-equivalences and $\Mc$-fibrations.
\end{defn}

\begin{thm}[$\Mc$-model structure for categories]\label{thm:M cat} 
  Let $\Mc$ be a set of finite, strongly connected categories.
  The $\Mc$-cofibrations, $\Mc$-fibrations and $\Mc$-equivalences form a 
  proper, cofibrantly generated model structure
  on the category of small categories, the {\em $\Mc$-model structure}.
  Every $\Mc$-cofibration is a retract of a Dwyer map. 
\end{thm}
\begin{proof}
We use the numbering of the model category axioms
as in~\cite[3.3]{dwyer-spalinski}. 
The category of categories is complete and cocomplete,
so axiom MC1 holds.
The $\Mc$-equivalences satisfy the 2-out-of-3 axiom MC2. 
The classes of $\Mc$-equivalences, $\Mc$-cofibrations and $\Mc$-fibrations 
are closed under retracts, so axiom MC3 holds.

One half of MC4 (lifting properties) holds by the definition of $\Mc$-cofibrations.
The proof of the remaining axioms uses Quillen's
small object argument, originally given in \cite[II p.\,3.4]{Q}, 
and later axiomatized in various places,
for example in \cite[7.12]{dwyer-spalinski} or \cite[Thm.\,2.1.14]{hovey-book}.
We let $\Delta[n]=\Delta(-,[n])$ denote the simplicial $n$-simplex,
$\partial\Delta[n]$ its boundary, and $\Delta[n,k]$ its $k$-th horn.
In Thomason's model structure, the set of inclusion functors
\[ i_n\colon  c(\Sd^2(\partial\Delta[n]))\ \to\ c(\Sd^2(\Delta[n])) \ ,  \]
for $n\geq 0$, detects fibrations that are also weak equivalences. 
By adjointness, the set
\begin{equation}\label{eq:I_for_F-proj_on_GT}
I_\Mc \ = \  \{ i_n\times I \colon  c(\Sd^2(\partial\Delta[n]))\times I
\to  c(\Sd^2(\Delta[n]))\times I \}_{n\geq 0, I\in\Mc}
\end{equation}
then detects $\Mc$-fibrations that are also $\Mc$-equivalences.
Similarly, the set of inclusions
$j_{k,n}: c(\Sd^2(\Delta[n,k]))\to c(\Sd^2(\Delta[n]))$,
for $n\geq 1$ and $0\leq k\leq n$, 
detects fibrations of categories; so by adjointness, the set
\[ J_\Mc \ = \  \{  j_{k,n}\times I \}_{n\geq 1, 0\leq k\leq n, I\in\Mc}  \]
detects $\Mc$-fibrations of categories.

All functors in $I_\Mc$ and $J_\Mc$ are Dwyer maps by 
\cite[Prop.\,4.2]{thomason:model cat cat} and Proposition \ref{prop:preserve Dwyer}.
The class of Dwyer maps is closed under coproducts,
composition \cite[Lemma 5.3 (i)]{thomason:model cat cat},
cobase change \cite[Prop.\,4.3]{thomason:model cat cat}
and sequential colimits of categories \cite[Lemma 5.3 (ii)]{thomason:model cat cat}.
Sources and targets of all functors in $I_\Mc$ and $J_\Mc$ 
are finite categories, i.e., they have finitely many objects and finitely many morphisms.
So these categories are finite with respect to sequential colimits of categories.

Now we can prove the factorization axiom MC5.
Every functor in $I_\Mc$ and $J_\Mc$ is an $\Mc$-cofibration by adjointness.
The small object argument applied to the set  
$I_\Mc$ gives a factorization of any functor as
an $\Mc$-cofibration followed by a functor with the right lifting property
with respect to $I_\Mc$. Since $I_\Mc$ detects the $\Mc$-acyclic fibrations,
this provides the factorizations as $\Mc$-cofibrations 
followed by $\Mc$-acyclic fibrations.

For the other half of the factorization axiom MC5
we apply the small object argument to the set $J_\Mc$; 
we obtain a factorization of any functor as
a $J_\Mc$-cell complex followed by a functor with the right lifting property
with respect to $J_\Mc$. Since $J_\Mc$ detects the $\Mc$-fibrations,
it remains to show that every $J_\Mc$-cell complex 
is an $\Mc$-equivalence.
To this end we observe that the morphisms in $J_\Mc$ are $\Mc$-equivalences
and Dwyer maps.
By Proposition \ref{prop:global equiv basics},
the property of being an $\Mc$-equivalence and a Dwyer map
is closed under composition, coproducts, cobase changes and sequential composites.
So every $J_\Mc$-cell complex is a Dwyer map and an $\Mc$-equivalence.

It remains to prove the other half of MC4, i.e., that every $\Mc$-cofibration $f:A\to B$
that is also an $\Mc$-equivalence has the left lifting property
with respect to $\Mc$-fibrations.
The small object argument provides a factorization 
\[  A \ \xra{\ j\ } \ W\ \xra{\ q\ } \ B\]
as a $J_\Mc$-cell complex followed by an $\Mc$-fibration.
In addition, $q$ is an $\Mc$-equivalence since $f$ and $j$ are.
Since $f$ is an $\Mc$-cofibration, a lifting in 
\[\xymatrix{
A \ar[r]^-j \ar[d]_f & W \ar[d]^q_(.6)\sim \\
B \ar@{=}[r] \ar@{..>}[ur] & B }\]
exists. Thus $f$ is a retract of the functor $j$ that has the left lifting
property for $\Mc$-fibrations. 
So $f$ itself has the left lifting property for $\Mc$-fibrations.
This completes the proof of the model category axioms.

Now we argue that every $\Mc$-cofibration $f:A\to B$ is a retract of a Dwyer map.
The small object argument provides a factorization 
\[  A \ \xra{\ i\ } \ X\ \xra{\ p\ } \ B\]
of $f$ as an $I_\Mc$-cell complex followed by a functor that is both an $\Mc$-fibration
and an $\Mc$-equivalence. The functor $i$ is then a Dwyer morphism because this class of
functors is closed under coproducts, cobase change, 
composition and sequential composition.
Since $f$ is an $\Mc$-cofibration, a lifting in 
\[\xymatrix{
A \ar[r]^-i \ar[d]_f & X \ar[d]^p_(.6)\sim \\
B \ar@{=}[r] \ar@{..>}[ur] & B }\]
exists. Thus $f$ is a retract of the Dwyer morphism $i$.

Right properness of the model structure is a straightforward consequence
of right properness of Thomason's non-equivariant model structure,
plus the fact that $\Fun(I,-)$ preserves pullbacks of categories
and takes $\Mc$-fibrations to non-equivariant fibrations.
For left properness we first observe that cobase change along Dwyer morphisms
preserves $\Mc$-equivalences by Theorem \ref{thm:Dwyer M pushout}.
Since the class of Dwyer morphisms is closed under cobase change,
the Dwyer morphisms are thus `flat' in the sense of 
\cite[Def.\,B.9]{hill-hopkins-ravenel}; flat morphisms are called
`h-cofibrations' in \cite[Def.\,1.1]{batanin-berger}.
As observed without proof in \cite[Prop.\,B.11 (iii)]{hill-hopkins-ravenel} 
and \cite[Lemma 1.3]{batanin-berger}, the class of flat morphism 
is closed under retracts. We honor tradition and also leave the proof of this
closure property to the reader. Since $\Mc$-cofibrations are retracts of Dwyer
morphisms, $\Mc$-cofibrations are flat. In particular, this proves left properness
of the $\Mc$-model structure.
\end{proof}

\begin{cor}\label{cor:equivalent}
  Let $\Mc$ be a set of finite, strongly connected categories 
  and $\Nc$ a subset of $\Mc$. 
  \begin{enumerate}[\em (i)]
  \item The identity functor of $\cat$
    is a right Quillen functor from the $\Mc$-model structure to the $\Nc$-model structure.
  \item Suppose that every category in $\Mc$ is homotopy equivalent to a category in $\Nc$. 
    Then the identity is a right Quillen equivalence 
    from the $\Mc$-model structure to the $\Nc$-model structure.  
  \item If $\Mc$ contains a non-empty category, then the identity functor of $\cat$
    is a right Quillen functor from the $\Mc$-model structure
    to the Thomason model structure.
  \end{enumerate}
\end{cor}
\begin{proof}
Part (i) is clear from the definitions.
For part (ii) we let $I$ be any category from $\Mc$ and $F:I\to J$ a homotopy equivalence to 
a category in $\Nc$. Then for every small category $\Cc$,
the functor $\Fun(F,\Cc):\Fun(J,\Cc)\to\Fun(I,\Cc)$ is a homotopy equivalence of
categories, natural in $\Cc$. So a functor is an $\Mc$-equivalence
if and only if it is an $\Nc$-equivalence. The claim follows.

(iii) We $I$ be a non-empty category in $\Mc$ and $i$ an object of $I$.
Then the `constant functor' functor and evaluation at $i$ are
natural functors
\[ \Cc \ \xra{\text{const}}\ \Fun(I,\Cc)\ \xra{\ \ev_i\ }\ \Cc \]
whose composite is the identity functor of $\Cc$.
So the simplicial set $\Ex^2(N \Cc)$ is a natural retract 
of $\Ex^2(N\Fun(I,\Cc))$. Since Kan fibrations 
and acyclic Kan fibrations of simplicial sets are closed under retracts,
this shows that every $\Mc$-fibration is a Thomason fibration,
and every $\Mc$-acyclic fibration is a Thomason acyclic fibration.
\end{proof}

\section{Categories versus simplicial functors}
\label{sec:cats vs sfunctors}

In this section we identify the $\Mc$-homotopy theory of small categories
with the homotopy theory of presheaves on a certain simplicial category $\bO_\Mc$
made from $\Mc$. 
The main result is Theorem \ref{thm:Mcat and Morbispace}, providing a Quillen equivalence
of model categories.
The result and its proof are reminiscent of Elmendorf's theorem \cite{elmendorf}
that identifies the genuine homotopy theory of $G$-spaces, for a topological group $G$,
with the homotopy theory of fixed point diagrams. Indeed, the categories $I$
in $\Mc$ play a role analogous to that of the coset spaces $G/H$
for the closed subgroups $H$ of $G$,
and the orbit category $\bO_\Mc$ plays a role similar to the orbit category of $G$.

Theorem \ref{thm:cofree} is an application of
the Quillen equivalence; it constructs {\em cofree categories} 
that are, informally speaking,
`rich in functors from strongly connected categories'.

\medskip

In the following we will make frequent use
of the fact that the functor categories $\Fun(I,J)$
provide an enrichment of the category $\cat$ of small categories in itself;
in other words, the category $\cat$ is underlying a 2-category.
This means concretely that for all small categories $I,J$ and $K$, composition of functors
extends to functor
 \[ \circ \ : \ \Fun(J,K)\times \Fun(I,J)\ \to\  \Fun(I,K) \ ,\]
and this extended composition is strictly associative and unital.

We use the term `simplicial category' as synonymous for 
`category enriched in the cartesian closed category of simplicial sets'.
These can be viewed as special kinds of simplicial objects
in the category of small categories, namely those whose simplicial
set of objects is constant.

\begin{defn}[Orbit category]\label{def:O_M}
  Let $\Mc$ be a class of finite, strongly connected categories.
  The {\em $\Mc$-orbit category} $\bO_\Mc$ is the simplicial category 
  whose objects are all categories in $\Mc$, and 
  where the simplicial set of morphisms from $I$ to $J$ is
  \[ \bO_\Mc(I,J)\ = \ N\Fun(I, J) \ , \]
  the nerve of the functor category.
  Composition in $\bO_\Mc$ is induced by composition of functors and natural transformations,
  i.e., defined as the composition
\begin{align*}
  \bO_\Mc(J,K)\times \bO_\Mc(I,J)\ &= \ 
  N\Fun(J,K)\times N\Fun(I,J)\\ 
&\iso\   N( \Fun(J,K)\times \Fun(I,J))\ \xra{N\circ} \ N\Fun(I,K)\ = \ \bO_\Mc(I,K) \ .
\end{align*}
  The isomorphism is the fact that the nerve functor preserves products.
\end{defn}

The orbit category $\bO_\Mc$ is defined to parameterize the natural structure
between the nerves of the functor categories $\Fun(I,\Cc)$ for 
a fixed category $\Cc$ and varying $I$ in $\Mc$. 
The following construction of the $\Mc$-nerve $N_\Mc\Cc$ makes this explicit.
We define an {\em $\bO_\Mc$-module} as a contravariant simplicial functor
from the simplicial category $\bO_\Mc$ to the category of simplicial sets.
We write
\[ \bO_\Mc\Mod \ = \ \Fun^{\text{simp}}(\bO^{\op}_\Mc,\sset) \]
for the category of $\bO_\Mc$-modules and simplicial natural transformations.

\begin{con}[$\Mc$-nerve]\label{con_M-nerve}
  Let $\Mc$ be a class of finite, strongly connected categories.
  The {\em $\Mc$-nerve} of a small category $\Cc$ 
  is the $\bO_\Mc$-module $N_\Mc\Cc:\bO_\Mc^{\op}\to \sset$ whose value at $I$ is
  \[ (N_\Mc\Cc)(I) \ = \ N\Fun(I,\Cc)\ , \]
  the nerve of the functor category. If $J$ is another category from $\Mc$,
  the structure morphism
  \[  (N_\Mc\Cc)(J)\times\bO_\Mc(I,J)\ \to \ (N_\Mc\Cc)(I) \]
  is the composite
  \begin{align*}
    (N_\Mc\Cc)(J)\times \bO_\Mc(I,J)\ &= \ 
    N\Fun(J,\Cc)\times   N\Fun(I,J)\\ &\iso \   N( \Fun(J,\Cc) \times\Fun(I,J))
    \ \xra{\ N\circ\ } \ N\Fun(I,\Cc)\ =\ (N_\Mc\Cc)(I) \ .
  \end{align*}
  We have used again that the nerve functor preserves products.
\end{con}

The definition of the $\Mc$-nerve is geared up so that
the $\Mc$-nerve of a category $J$ from $\Mc$ is the simplicial
functor represented by $J$, i.e.,
\begin{equation}\label{eq:J represents}  
 N_\Mc J  \ = \ \bO_\Mc(-,J)\ .
\end{equation}

\begin{rem}
The $\Mc$-nerve functor actually factors through the ordinary nerve as the composite
\[ \cat \ \xra{\ N \ }\ \sset \ \xra{\map(N\Mc,-)} \ \bO_\Mc\Mod\ . \]
The functor $\map(N\Mc,-)$ takes a simplicial set $A$ 
to the $\bO_\Mc$-module with values
\[ \map(N\Mc,A)(I)\ = \ \map(N I,A)\ , \]
the simplicial set of  morphism from $N I$ to $A$.
If $J$ is another category from $\Mc$, the structure morphism
\[  \map(N\Mc,A)(J)\times\bO_\Mc(I,J)\ \to \ \map(N\Mc,A)(I) \]
  is the composite
  \begin{align*}
    \map(N J,A)\times  N\Fun(I,J)\ \iso \ \map(N J,A)\times  \map(N I,N J)\ \xra{\ \circ\ } \ \map(N I,A) \ .
  \end{align*}
  The isomorphism is the fact that the nerve functor is fully faithful
  in the enriched sense, i.e., endowed with natural isomorphisms of simplicial sets
  \[ \bO_\Mc(I,J)\ = \  N \Fun(I,J) \ \iso \ \map(N I,N J)\ .\]
  The isomorphism between $N_\Mc\Cc$ and the composite $\map(N\Mc,-)\circ N$
  is also an instance of the enriched fully-faithfulness, i.e., its value
  at $I\in\Mc$ is the isomorphism
  \[ (N_\Mc \Cc)(I)\ = \  N \Fun(I,\Cc) \ \iso \ \map(N I,N \Cc)\ .\]
  These isomorphisms are compatible with the $\bO_\Mc$-module structure.
\end{rem}

\begin{eg}[Posets]\label{eg:posets}
Every space is weakly equivalent to the nerve of a small category,
and this small category can even be chosen to be a poset.
More is actually true: 
Raptis \cite[Thm.\,2.6]{raptis} shows that Thomason's model structure
on the category of small categories restricts to the
full subcategory $\poset$ of poset categories, 
and that moreover the inclusion $\poset\to\cat$
is a Quillen equivalence.
In particular, poset morphisms are weak equivalences if and only
if they induce weak homotopy equivalences on nerves.
In combination with Thomason's result, this implies that the nerve functor
$N:\poset\to\sset$ is an equivalence of homotopy theories.
  
However, posets categories are very special
from our present global perspective, as their $\Mc$-nerves are always constant.
Indeed, every functor from a strongly connected category to a poset category $P$
is constant, i.e., its image consists of a single object and its identity.
So for every $I$ in $\Mc$, the `constant functor' functor
\[ \text{const}\ : \ P\ \to \ \Fun(I,P) \]
is an isomorphism of categories, and the morphism of nerves
\[ N\text{const}\ : \ N P\ \to \ N\Fun(I,P) \]
is an isomorphism of simplicial sets. 
Put yet another way, the $\Mc$-nerve $N_\Mc P$ of a poset category
is the constant functor with value the nerve $N P$.
\end{eg}

\begin{eg}[Grothendieck constructions]\label{eg:Grothendieck construction}
A functor $F:K\to\cat$ from some small indexing category to the
category of small categories gives rise to a new category, 
the {\em  Grothendieck construction} $K \smallint F$. Objects of this category
are pairs $(y,k)$ where $k$ is an object of $K$ and $y$ is an object of $F(k)$.
Morphisms from $(x,j)$ to $(y,k)$ are all pairs $(\psi,f)$
where $f:j\to k$ is a morphism in $K$ and $\psi:F(f)(x)\to y$ is a morphism
in $F(k)$. Composition is defined by
\[ (\varphi,g)\circ (\psi,f)  \ = \ (\varphi\circ F(g)(\psi), g\circ f) \ .\]
Thomason shows in \cite[Thm.\,1.2]{thomason:hocolim} that the   
Grothendieck construction models the homotopy colimit. More precisely,
Thomason exhibits a chain of two natural weak equivalences between
the nerve of the category $K\smallint F$ and the
homotopy colimit, in the sense of Bousfield and Kan \cite[Ch.\,12]{bousfield-kan},
of the functor
\[ N\circ F\ : \ K \ \to \ \sset\ . \]
This result does {\em not} extend literally to the $\Mc$-nerve.
To illustrate this, we consider another small category $I$.
For varying $k$ in $K$, the evaluation functors $\ev_{I,F(k)}:\Fun(I,F(k))\times I\to F(k)$
provide a natural transformation of category valued functors
\[ \ev_{I,F}\ : \ \Fun(I,F)\times I \ \Longrightarrow\ F \ : \ K \ \to \ \cat\ . \]
This induces a functor on Grothendieck constructions
\[ K\smallint (\Fun(I,F)\times I) \ \to\ K\smallint F \ . \]
The canonical functor
\[ K\smallint (\Fun(I,F)\times I) \ \to\ \left(K\smallint \Fun(I,F)\right)\times I \]
is an isomorphism of categories, so altogether we obtain an evaluation functor
\[ \left(K\smallint \Fun(I,F)\right)\times I\ \to \ K\smallint F \ ,\quad
(k,\psi:I\to F(k),i)\ \longmapsto \ (k, \psi(i)) \ .\]
The adjunction between $-\times I$ and $\Fun(I,-)$
turns this evaluation functor into a functor
\begin{equation}  \label{eq:int_vs_Fun}
 K\smallint \Fun(I,F)\ \to \ \Fun(I,K\smallint F) \ .  
\end{equation}
This functor is {\em not} generally a weak equivalence, not even
if $I$ is finite and strongly connected.

However, the situation changes if the indexing category $K$
is a poset category. Indeed, if $K$ is a poset and $I$ is strongly connected,
then every functor $\psi:I\to K$ is constant, so every functor
$I\to K\smallint F$ is concentrated at a single object of $K$ and its identity.
So then the functor \eqref{eq:int_vs_Fun} is an isomorphism of categories. 
\end{eg}

\begin{con}
As is already clear from Thomason's work, the nerve functor has to be tweaked
if we want it to be right adjoint to a homotopically meaningful functor.
Fortunately, the direct analog of Thomason's patch also works in our context:
we simply postcompose with $\Ex^2$, the twofold iterate of Kan's functor 
$\Ex$ \cite[Sec.\,3]{kan:on css}.
In more detail, we define a functor
\[ \Ex^2\ : \  \bO_\Mc\Mod \ \to \ \bO_\Mc\Mod \]
on the category of $\bO_\Mc$-modules by
\[ (\Ex^2 Y)(I)\ = \ \Ex^2(Y(I))\ . \]
The structure morphism
  \[  (\Ex^2 Y)(J)\times\bO_\Mc(I,J)\ \to \ (\Ex^2 Y)(I) \]
  is the composite
  \begin{align*}
    \Ex^2( Y(J) ) \times \bO_\Mc(I,J)\ &\xra{\Id\times \kappa} \ 
    \Ex^2( Y(J) ) \times \Ex^2( \bO_\Mc(I,J))\\ 
    &\iso \ \Ex^2\left( Y(J) \times \bO_\Mc(I,J)\right)\ 
    \ \xra{\ \Ex^2\circ\ } \ \Ex^2(Y(I)) \ .
  \end{align*}
  Here $\kappa:\Id\to\Ex^2$ is the iterate of the natural weak equivalence $e K:K\to\Ex K$
  defined in \cite[p.\,453]{kan:on css},
and the isomorphism is the fact that $\Ex^2$ preserves products.
For varying $I$ in $\Mc$, the weak equivalences $\kappa_{Y(I)}:Y(I)\to\Ex^2(Y(I))$ 
define a weak equivalence of $\bO_\Mc$-modules
\[ \kappa_Y\ : \ Y\ \to \ \Ex^2 Y\ . \]
\end{con}

\begin{prop}
  Let $\Mc$ be a class of finite, strongly connected categories.  
  Then the functor
  \[ \Ex^2\circ N_\Mc\ : \ \cat \ \to \ \bO_\Mc\Mod \]
  has a left adjoint $\Gamma$. Moreover, the left adjoint can be chosen so that
  \begin{equation}\label{eq:normalize Gamma}
    \Gamma(A\times\bO_\Mc(-,J))\ =\ c(\Sd^2 A)\times J 
  \end{equation}
  for all simplicial sets $A$ and all $J\in \Mc$.
\end{prop}
\begin{proof}
  To simplify the notation we write $\bO=\bO_\Mc$, i.e., we drop the
  subscript `$\Mc$'. We need to show that for every $\bO$-module $Y$ the functor
  \[  \cat \ \to \ \sets\ , \quad \Cc \ \longmapsto \ 
  \bO\Mod(Y,\Ex^2(N_\Mc\Cc)) \]
  is representable by a small category. Any choices of representing data
  will then canonically assemble into a left adjoint functor.

  If $Y=A\times\bO(-,J)$ for a simplicial set $A$ and a category $J\in \Mc$, then
  \begin{align*}
    \bO\Mod(A\times\bO(-,J),\Ex^2(N_\Mc\Cc)) \ &\iso \
    \bO\Mod(\bO(-,J),\map(A,\Ex^2(N_\Mc\Cc))) \\ 
    &\iso \ \sset(A, \Ex^2(N_\Mc\Cc)(J) )\\     
    &= \ \sset(A, \Ex^2( N\Fun(J,\Cc) ))\\     
    &\iso \ \cat(c(\Sd^2 A), \Fun(J,\Cc) )\\     
    &\iso \ \cat(c(\Sd^2 A)\times J, \Cc )\ .     
  \end{align*}
  So the representability property holds for such $\bO$-modules, and
  we can normalize the left adjoint $\Gamma$ to satisfy \eqref{eq:normalize Gamma}.

  A general $\bO$-module $Y$ is a coequalizer of two morphisms
  \begin{equation}\label{eq:coequalizer}    
    \xymatrix@C=10mm{  
      \coprod_{I,J\in\Mc}     Y(J)\times \bO(I,J)\times \bO(-,I)  
      \  \ar@<.4ex>[r] 
      \ar@<-.4ex>[r] & \ \coprod_{I\in\Mc} Y(I)\times \bO(-,I) }  \ .  
  \end{equation}
  On the $(I,J)$-summand, 
  one of the two morphisms is the product of the action morphism
  $Y(J) \times \bO(I,J) \to Y(I)$
  with the identity of the represented $\bO$-module $\bO(-,I)$.
  The other morphism is the product of the identity of $Y(J)$ with the
  composition morphism of $\bO$-modules $\bO(I,J)\times\bO(-,I)\to\bO(-,J)$.
  The category of small categories has coproducts, so source and
  target of the two morphisms satisfy the representability property, namely by the
  categories 
  \[   {\coprod}_{I,J\in\Mc}\,    c(\Sd^2( Y(J)\times \bO(I,J)))\times I  
  \text{\quad and\quad}  {\coprod}_{I\in\Mc}\, c(\Sd^2 Y(I))\times I \ ,  \]
  respectively. The two morphisms in the coequalizer diagram \eqref{eq:coequalizer}
  correspond to unique functors between the representing categories.
  Since the category of small categories has coequalizers, we can define
  a small category $\Gamma Y$ as a coequalizer in $\cat$ of the corresponding
  functors. The verification that this category indeed represents
  the functor $\bO\Mod(Y,\Ex^2(N_\Mc -))$ uses the fact that
  $\Ex^2\circ N_\Mc$ preserves limits.
\end{proof}

As an enriched functor category, 
the category $\bO_\Mc\text{-mod}$ has a projective  
model structure, see \cite[Thm.\,6.1 (1)]{schwede-shipley:equivalences}.
The weak equivalences or fibrations are those morphisms $f:X\to Y$
such that for every $I$ in $\Mc$,
the morphism of simplicial sets $f(I):X(I)\to Y(I)$ 
is a weak equivalence or Kan fibration, respectively.

\begin{thm} \label{thm:Mcat and Morbispace}
  Let $\Mc$ be a set of finite, strongly connected categories.
  The adjoint functor pair
\[\xymatrix@C=12mm{ 
  \Gamma \ : \ \bO_\Mc\text{\em -mod}  \quad \ar@<.4ex>[r] & 
\quad \cat\ : \ \Ex^2\circ N_\Mc \ar@<.4ex>[l]  }\]  
is a Quillen equivalence between the category
of small categories with the $\Mc$-model structure
and the category of $\bO_\Mc$-modules with the projective model structure.
\end{thm}
\begin{proof}
  We again abbreviate $\bO_\Mc$ to $\bO$. 
  The fact that $(\Gamma,\Ex^2\circ N_\Mc)$ is a Quillen functor pair
  is straightforward.
  Indeed, the $\Mc$-fibrations of small categories are defined so that
  the right adjoint functor $\Ex^2\circ N_\Mc$
  sends them to fibrations in the projective model structure of $\bO$-modules.
  The right adjoint also takes $\Mc$-equivalences of small categories
  to objectwise weak equivalences of $\bO$-modules, by the definition
  of `$\Mc$-equivalences' and the fact that the 
  natural transformation $\kappa_K:K\to \Ex^2 K$ is a weak equivalence of simplicial
  sets.   Even better: $\Ex^2\circ N_\Mc$ preserves and reflects weak equivalences.
  So $\Ex^2\circ N_\Mc$ is in particular a right Quillen functor,
  and $(\Gamma,\Ex^2\circ N_\Mc)$ is a Quillen functor pair.

  A more substantial argument is needed to see that
  the Quillen functor pair $(\Gamma,\Ex^2\circ N_\Mc)$ is a Quillen equivalence.
  The key step is the verification that for every cofibrant $\bO$-module $Y$,
  the adjunction unit $Y\to\Ex^2(N_\Mc(\Gamma Y))$ is an $\Mc$-equivalence.
  To this end we consider the class $\Oc$ of $\bO$-modules for which the
  adjunction unit is a weak equivalence.
  The class $\Oc$ contains all $\bO$-modules of the form $A\times\bO(-,J)$
  for a simplicial set $A$ and $J\in\Mc$.
  Indeed, under the identifications \eqref{eq:J represents} and \eqref{eq:normalize Gamma}, 
  the unit 
  \[ A\times\bO(-,J)\ \to \ \Ex^2(N_\Mc(\Gamma(A\times\bO(-,J)))) \]
  becomes the morphism of $\bO$-modules
  \[ \eta_A\times \kappa_{\bO(-,J)}\ : \
  A\times\bO(-,J)\ \to \ \Ex^2(N(c(\Sd^2 A)))\times \Ex^2\bO(-, J) \ , \]
  the product of the adjunction unit $A\to \Ex^2(N(c(\Sd^2 A)))$
  and the weak equivalence of $\bO$-modules $\kappa_{\bO(-,J)}$.
  Since $\eta:\Id\to \Ex^2\circ N\circ c\circ \Sd^2$ and $\kappa:\Id\to \Ex^2$
  are natural weak equivalences of simplicial sets, this proves that $A\times \bO(-,J)$
  belongs to the class $\Oc$.

  The functors $\Gamma$, $N_\Mc$ and $\Ex^2$ all commute with coproducts (i.e., disjoint unions);
  for $N_\Mc$ this uses the hypothesis that all categories in $\Mc$ are connected.
  Since weak equivalences of simplicial sets are closed under coproducts,
  this proves that the class $\Oc$ is closed under coproducts of $\bO$-modules.
  
  Now we consider an index set $S$, an $S$-indexed family $I_s$ of categories in $\Mc$,
  numbers $n_s\geq 0$ and a pushout square of $\bO$-modules:
  \begin{equation}\begin{aligned}\label{eq:initial_pushout}
      \xymatrix{ 
        \coprod_{s\in S} \partial\Delta[n_s]\times \bO(-,I_s)
        \ar[r]\ar[d] &
        \coprod_{s\in S} \Delta[n_s]\times \bO(-,I_s)  \ar[d] \\
        X \ar[r] & Y }
  \end{aligned}\end{equation}
  We claim that if $X$ belongs to $\Oc$, then so does $Y$.  
  As a left adjoint, $\Gamma$ preserves coproducts and pushout.
  Thus $\Gamma$ takes the square \eqref{eq:initial_pushout} 
  to a pushout square of categories:
  \[  \xymatrix{ 
    \coprod_{s\in S}  c(\Sd^2\partial\Delta[n_s])\times  I_s
    \ar[r]\ar[d] &
    \coprod_{s\in S}  c(\Sd^2\Delta[n_s])\times I_s    \ar[d] \\
    \Gamma X \ar[r] & \Gamma Y }  \]
  Here we exploited the relation \eqref{eq:normalize Gamma}.
  Since the inclusion $ c(\Sd^2\partial\Delta[n_s])\to c(\Sd^2\Delta[n_s])$
  is a Dwyer map \cite[Prop.\,4.2]{thomason:model cat cat},
  the upper horizontal functor in this square is a Dwyer map.
  So the square
  \[ \xymatrix{ 
    \coprod_{s\in S} N_\Mc ( c(\Sd^2\partial\Delta[n_s])\times I_s)\ar[r]\ar[d] &
    \coprod_{s\in S} N_\Mc ( c(\Sd^2\Delta[n_s])\times I_s )\ar[d] \\
    N_\Mc(\Gamma X) \ar[r] & N_\Mc( \Gamma Y) }\]
  is a homotopy pushout square of $\bO$-modules by Proposition \ref{prop:global equiv basics} (iv).
  The functor $\Ex^2$ commutes with coproducts and receives a natural weak equivalence
  from the identity, so the square
  \[ \xymatrix@R=5mm{ 
    \coprod_{s\in S} \Ex^2 (N_\Mc ( \Gamma(\partial\Delta[n_s]\times\bO(-, I_s))))\ar[r]\ar@{=}[d] &
    \coprod_{s\in S} \Ex^2(N_\Mc ( \Gamma(\Delta[n_s]\times \bO(-,I_s)))) \ar@{=}[d] \\
        \coprod_{s\in S} \Ex^2 (N_\Mc ( c(\Sd^2\partial\Delta[n_s]) \times I_s))\ar[r]\ar[dd] &
    \coprod_{s\in S} \Ex^2(N_\Mc ( c(\Sd^2\Delta[n_s])\times I_s)) \ar[dd] \\
    &&\\
    \Ex^2(N_\Mc(\Gamma X)) \ar[r] & \Ex^2(N_\Mc( \Gamma Y)) }\]
  is another homotopy pushout square of $\bO$-modules.
  The adjunction units induce compatible maps from the
  original pushout square \eqref{eq:initial_pushout} to this last square.
  Since $X$ belongs to $\Oc$ by hypothesis and the two upper $\bO$-modules in 
  \eqref{eq:initial_pushout} belong to $\Oc$ by the previous paragraphs,
  also $Y$ belongs to $\Oc$.

  Now we consider a colimit $Y$ of a sequence of morphisms of $\bO$-modules
  \[ Y_0 \ \to \ Y_1 \ \to \ \cdots \ \to \ Y_n \ \to \ \cdots\ . \]
  We claim that if all $Y_n$ belong to $\Oc$,
  then so does the colimit $Y$. Indeed, the functor $\Gamma$ preserves arbitrary colimits,
  and $N_\Mc$ and $\Ex$ preserve sequential colimits; for $N_\Mc$ this uses the hypothesis
  that all categories in $\Mc$ are finite. Colimits of $\bO$-modules are formed objectwise,
  and sequential colimits of simplicial sets are fully homotopical, so this proves the claim.

  The set $\Ic$ of inclusions
  \[   \partial\Delta[n]\times \bO(-,I) \ \to \ \Delta[n] \times \bO(-,I) \ , \]
  for $n\geq 0$ and $I\in \Mc$,
  generates the cofibrations of the projective model structure on $\bO$-modules.
  The previous closure properties of the class $\Oc$ imply that every $\bO$-module
  that can be built as a sequential colimit of cobase changes 
  of coproducts of morphisms in $\Ic$ belongs to $\Oc$. 
  By the small object argument, every cofibrant $\bO$-module is a retract
  of such a special $\bO$-module. Since the class $\Oc$ is also closed under retracts,
  it contains all cofibrant $\bO$-modules.
  
  Now we can complete the proof.
  The right Quillen functor $\Ex^2\circ N_\Mc$ preserves and detects $\Mc$-equivalences,
  and for every cofibrant $\bO$-module $Y$,
  the adjunction unit $Y\to\Ex^2(N_\Mc(\Gamma Y))$ is an $\Mc$-equivalence.
  So the pair $(\Gamma,\Ex^2\circ N_\Mc)$ is a Quillen equivalence,
  for example by \cite[Cor.\,1.3.16]{hovey-book}. 
\end{proof}

\begin{rem}[Generalization to infinite categories]\label{rk:infinite} 
  Our results work more generally for strongly connected categories that are
  not necessarily finite;
  all we need is that there is a {\em set} of representatives for the isomorphism classes
  of the categories in $\Mc$. 
  Effectively, this means that we impose a cardinality bound
  on the categories in the class $\Mc$.
  We only indicate what goes into this, and leave the details to interested readers.

  For any set $\Mc$ of strongly connected categories,
  the definitions of $\Mc$-equivalences, $\Mc$-fibrations and $\Mc$-cofibrations
  make perfect sense without a size restriction, 
  and they form a cofibrantly generated, proper 
  model structure on the category of small categories.
  Moreover, the functor $\Ex^2\circ N_\Mc$ is a right Quillen equivalence 
  to the category of $\bO_\Mc$-modules.
  Indeed, the proofs of Theorems \ref{thm:M cat} and \ref{thm:Mcat and Morbispace}
  generalize almost unchanged.
  There is one caveat, however: if the class $\Mc$ contains an infinite category,
  then $\Fun(I,-)$ and $N_\Mc$ need not commute with sequential colimits.
  So we cannot use the countable version of the small object argument
  for generalizing the proofs of Theorem \ref{thm:M cat}
  and Theorem \ref{thm:Mcat and Morbispace}.
  However, since there
  is a cardinality bound for the categories in $\Mc$, we can choose a sufficiently
  large regular cardinal and then employ a transfinite version of the small
  object argument, for example as in \cite[Thm.\,2.1.14]{hovey-book}.
\end{rem}

In this rest of this section we give an application of
the Quillen equivalence of Theorem \ref{thm:Mcat and Morbispace}:
we construct {\em cofree categories} that are, informally speaking,
`rich in functors from strongly connected categories'.
To motivate this concept, we let $I$ and $\Cc$ be small categories.
In general, one cannot hope that every continuous map $|N I|\to |N \Cc|$
is homotopic to $|N F|$ for some functor $F:I\to\Cc$.
Cofree categories are defined by requiring a strong version of this realization property,
namely that a specific map \eqref{eq:M cofree map}
from the realization of the functor category $\Fun(I,\Cc)$
to the space of maps from  $|N I|$ to $|N \Cc|$ is a weak equivalence.
We will show in Theorem \ref{thm:cofree} below
that for every given set $\Mc$ of finite, strongly connected categories, every homotopy
type can be represented by a small category that is $\Mc$-cofree.

\begin{con}
Let $I$ and $\Cc$ be small categories.
The evaluation functor 
\[ \ev_{I,\Cc}\ :\ \Fun(I,\Cc)\times I\ \to\ \Cc \]
and the fact that the nerve
and geometric realization commute with products produce a continuous map
\[
|N\Fun(I,\Cc)|\times |N I| \ \iso \ |N(\Fun(I,\Cc)\times I)| \ \xra{|N\ev_{I,\Cc}|} \   |N \Cc| \ .
\]
Adjoint to this is a continuous map
\begin{equation}\label{eq:M cofree map}
  |N\Fun(I,\Cc)| \ \to \ \map(|N I|, |N \Cc|) \ .
\end{equation}
Here $\map(-,-)$ is the internal mapping space in the category of compactly generated
spaces \cite{mccord}, i.e., the set of continuous maps with the Kelleyfied compact-open topology. 
\end{con}

\begin{defn}
  Let $\Mc$ be a class of small categories.
  A small category $\Cc$ is {\em $\Mc$-cofree} if the map \eqref{eq:M cofree map}
  is a weak equivalence for every category $I$ in $\Mc$.
\end{defn}

\begin{eg}[Groupoids]
  Groupoids are examples of cofree categories.
  To see this we use the fact that for every simplicial set $A$ and every Kan complex $Z$,
  the map
  \begin{equation}\label{eq:second_adjoint}
    |\map(A,Z)|\ \to\ \map(|A|,|Z|)
  \end{equation}  
  adjoint to
  \[ |\map(A,Z)|\times |A|\ \iso \ |\map(A,Z)\times A|\ \xra{\ |\ev_{A,Z}|\ } \ |Z| \]
  is a weak equivalence.
  Moreover, the nerve functor is fully faithful, even in the enriched sense 
  that the morphism 
  \[  N\Fun(I,J)\ \to \ \map(N I,N J)\]
  is an isomorphism for all small categories $I$ and $J$.
  The nerve of every groupoid $\Gc$ is a Kan complex,
  so the map \eqref{eq:M cofree map} factors as the composite
  \[  |N\Fun(I,\Gc)| \ \xra{\ \iso \ } \ |\map(N I,N\Gc)| \ \xra{\ \simeq \ } \
   \map(|N I|, |N \Gc|)  \]
  of a homeomorphism and a weak equivalence. 
  We conclude that groupoids are $\Mc$-cofree for all $\Mc$.
\end{eg}

\begin{thm}\label{thm:cofree}
  Let $\Mc$ be a set of strongly connected small categories.
  Then for every space $X$ there exists an $\Mc$-cofree category $\Cc$ whose nerve is
  weakly equivalent to $X$.
\end{thm}
\begin{proof}
  We may assume without loss of generality that $\Mc$ contains a terminal category
  $\ast$ with a single object and its identity; if that is not already the case,
  then we simply replace $\Mc$ by $\Mc\cup\{\ast\}$.
  Then for every $\bO_\Mc$-module $Y$ and every $I\in\Mc$, the composite map
  \[ |N I|\times |Y(I)|\ \iso \ |\bO_\Mc(\ast,I)\times Y(I)|\ \xra{|\text{act}|} \ 
  |Y(\ast)| \]
  is adjoint to a continuous map
  \[ |Y(I)|\ \to \ \map(|N I|,|Y(\ast)|)\ .    \]
  We call the $\bO_\Mc$-module $Y$ {\em cofree} if this map
  is a weak equivalence for all $I\in \Mc$. 
  Then a category $\Cc$ is $\Mc$-cofree if and only if its $\Mc$-nerve $N_\Mc\Cc$ 
  is cofree. Moreover, this notion of cofreeness is invariant under weak equivalences
  of $\bO_\Mc$-modules.

  Now we return to the given topological space $X$ and consider the $\bO_\Mc$-module 
  \[ R_X\ : \ \bO_\Mc^{\op} \ \to \ \sset \ , \quad R_X(I)\ =\ \map(N I,\text{sing}(X))\]
  represented by the singular complex of the space $X$.
  Since the singular complex is a Kan complex, the map \eqref{eq:second_adjoint}
  \[ |\map(N I,\text{sing(X)})|\ \to\ \map(|N I|,|\text{sing}(X)|)  \]
  is a weak equivalence.
  So the represented $\bO_\Mc$-module $R_X$ is cofree.

  The Quillen equivalence of Theorem \ref{thm:Mcat and Morbispace},
  or rather its generalization without the finiteness hypothesis
  in Remark \ref{rk:infinite}, provides a small category $\Cc$ and a chain of 
  weak equivalences of $\bO_\Mc$-modules between $R_X$ and the $\Mc$-nerve $N_\Mc\Cc$. 
  In particular, the nerve of the category $\Cc$ is weakly equivalent to
  the singular complex of $X$. Hence the space $|N\Cc|$ is weakly equivalent to $X$.
  Since $R_X$ is cofree, so is $N_\Mc \Cc$.
  Hence $\Cc$ is $\Mc$-cofree, and this completes the argument.
\end{proof}

\begin{eg}
At this point we recall that every path connected CW-complex
is homotopy equivalent to the nerve of a strongly connected small category.
Even better, McDuff's celebrated theorem \cite{mcduff} says that
such a space is homotopy equivalent to the classifying space of
a monoid $M$, i.e., the nerve of the category $B M$
with a single object whose endomorphism monoid is $M$. 
The functor category $\Fun(B M,\Cc)$ is isomorphic to the category
$M\Cc$ of $M$-objects in $\Cc$, compare Construction \ref{con:monoid actions} below.
Theorem \ref{thm:cofree} says that for every space $X$ the space $\map(|N(B M)|,X)$
can be modeled categorically by $M$-objects: 
there is a small category $\Cc$ 
and a weak equivalence $|N\Cc|\simeq X$, such that moreover the map
\eqref{eq:M cofree map}
\[ |N(M\Cc)| \ \to \  \map(|N(B M)|, |N \Cc|)  \ \simeq \ \map(|N(B M)|, X)  \]
is a weak equivalence.

For the sake of concreteness,
we illustrate this with a specific example.
McDuff's theorem predicts the existence of a monoid whose
classifying space is homotopy equivalent to the 2-sphere $S^2$.
Fiedorowicz \cite{fiedorowicz}
provided an explicit monoid $M$ with this property:
it is generated by two elements $a$ and $b$ subject to the relations
\begin{equation}\label{eq:fiedo}
 a^2 = a  = a b a \text{\qquad and\qquad} b^2 = b =  b a b \ . 
\end{equation}
In particular, $M=\{1,a,b,a b, b a\}$ has five elements.
Given any space $X$, our general theory provides a category $\Cc$ and a weak equivalence
$|N \Cc| \simeq X$  such that moreover the map \eqref{eq:M cofree map}
\[ |N(M\Cc)| \ \to \  \map(|N(B M)|, |N \Cc|)  \ \simeq \ \map(S^2, X)  \]
is a weak equivalence. 
The upshot is that the homotopy type of the space $\map(S^2, X)$ is encoded in 
the category of $\Cc$-objects equipped with two endomorphisms $a$ and $b$ 
that satisfy the relations \eqref{eq:fiedo}.
\end{eg}

\section{Global equivalences, orbispaces, and complexes of groups}
\label{sec:group case}

In this section we specialize the general results of the previous sections
to the class of classifying categories of finite groups;
this special case was the original motivation for the paper, and we refer to
it as the {\em global homotopy theory} of small categories.
In this situation, various special features show up, and
the global homotopy theory of categories is equivalent to the
homotopy theory of orbispaces in the sense of Gepner and Henriques \cite{gepner-henriques}.
Moreover, the cofibrant categories all arise from `complexes of groups' in the sense of
Haefliger \cite{haefliger}.

As we explain in Remark \ref{rk:infinite}, the restriction to {\em finite}
groups is not essential, and one can instead work relative to any class of groups
that has a set of isomorphism classes. For example, the results of this
section have analogs for the class of countable groups.

\begin{con}[Equivariant objects and classifying categories]\label{con:monoid actions}
We let $M$ be a monoid.
We are mostly interested in monoids that are in fact groups,
but various parts of the theory work without inverses.

An {\em $M$-object} in a category $\Cc$
is a pair $(x,\rho)$ consisting of an object $x$ of $\Cc$ and a monoid homomorphism
$\rho:M\to \Cc(x,x)$ to the endomorphism monoid. 
We often leave the homomorphism implicit and denote the $M$-object only by $x$.
A morphism of $M$-objects is a $\Cc$-morphism $f:x\to y$ 
that commutes with the $M$-actions, i.e., 
\[ \rho_y(m)\circ f \ = \ f\circ \rho_x(m) \]
for all $m\in M$.
We denote by $M\Cc$ the category of $M$-objects in $\Cc$.

The unique object of the classifying category $B M$ has a tautological $M$-action,
by the identity homomorphism of $M$. So the pair $(\ast,\Id_M)$
is a tautological $M$-object in $B M$.
Moreover, for every category $\Cc$,
evaluation at the unique object is an isomorphism of categories
\[ \Fun(B M,\Cc)\ \to \ M\Cc \ , \quad F\ \longmapsto \ F(\ast,\Id_M) \]
from the functor category to the category of $M$-objects in $\Cc$.
In the following, we tacitly identify the functor category $\Fun(B M,\Cc)$
with the category $M\Cc$ via this isomorphism.
\end{con}

\begin{defn}
  A functor between small categories $\Phi:X\to Y$ is a {\em global equivalence} 
  or {\em global fibration}
  if the induced functor $G \Phi:G X\to G Y$ on $G$-objects is a weak equivalence
  or Thomason fibration of categories for all finite groups $G$.
\end{defn}

If we specialize Theorem \ref{thm:M cat} to the class of
classifying categories of finite groups, we obtain the following result.

\begin{thm}[Global model structure for categories]\label{thm:global cat}
   The global equivalences and global fibrations are part of a proper,
  cofibrantly generated model structure
  on the category of small categories, the {\em global model structure}.
\end{thm}

Now we explain how the global homotopy theory of small categories
is the homotopy theory of orbispaces with finite isotropy groups.

\begin{con}[Global orbit category]\label{con:global orbit category}
  The {\em global orbit category} $\bO_{\gl}$ is the simplicial category
  whose objects are all finite groups,
  and where the simplicial set $\bO_{\gl}(K,G)$ of morphisms is
  \[ \bO_{\gl}(K,G)\ = \ N\Fun(B K,B G) \ , \]
  the nerve of the category of functors from $B K$ to $B G$.
  In other words,  $\bO_{\gl}$ is a special case of the orbit category
  as defined in Definition \ref{def:O_M}, for $\Mc$ the class of classifying categories
  $B G$ for all finite groups $G$.

  In this special case, the categories $\Fun(B K,B G)$ and the simplicial sets $\bO_{\gl}(K,G)$
  can be made more explicit. Indeed, every functor $B K\to B G$ is of the form
  $B \alpha$ for a unique group homomorphism $\alpha:K\to G$. 
  Every natural transformation $B\alpha\Longrightarrow B\beta$ 
  for two homomorphisms $\alpha,\beta:K\to G$ is given by a unique
  element $g\in G$ such that $\beta=c_g\circ\alpha$, where $c_g(\gamma)=g\gamma g^{-1}$
  is the inner automorphism specified by $g$.
  In particular, the categories $\Fun(B K, B G)$ are groupoids.
  Hence the simplicial sets $\bO_{\gl}(K,G)$ are 1-types, i.e.,
  they have trivial homotopy groups in dimensions bigger than 1.
  The path components $\pi_0( \bO_{\gl}(K,G))$ 
  are in bijection with conjugacy classes of group homomorphisms. 
  The fundamental group of $\bO_{\gl}(K,G)$ based at a homomorphism $\alpha:K\to G$
  is the centralizer of the image of $\alpha$.
  A different way to say this is that $\bO_{\gl}(K,G)$ is a disjoint union, 
  indexed by conjugacy classes of homomorphisms $\alpha:K\to G$, 
  of classifying spaces of the centralizer of the image of $\alpha$.
\end{con}  

\begin{defn}
  An {\em orbispace} is an $\bO_{\gl}$-module, i.e., a contravariant  simplicial functor
  from $\bO_{\gl}$ to the category of simplicial sets.
  We denote the category of orbispaces and simplicial natural transformations by
  $\bO_{\gl}\Mod$.
\end{defn}

For the class of classifying categories of finite groups, Construction \ref{con_M-nerve}
specializes to the {\em global nerve functor}
\[ N_{\gl} \ : \ \cat \ \to \ \bO_{\gl}\Mod\]
from small categories to orbispaces. The value of $N_{\gl}\Cc$ at a finite group $G$ is
\[ (N_{\gl}\Cc)(G) \ = \ N(G\Cc)\ ,\]
the nerve of the category of $G$-objects in $\Cc$; as before, we implicitly identify
$\Fun(B G,\Cc)$ with $G\Cc$ by evaluation at the tautological $G$-object in $B G$.
For every group $K$ we have
\[  N_{\gl} (B K)  \ = \ \bO_{\gl}(-,K)\ , \]
the restriction of the represented functor 
to $\bO_{\gl}$, compare \eqref{eq:J represents}.  
If we specialize Theorem \ref{thm:Mcat and Morbispace} to the class of
classifying categories of finite groups, we obtain the following result.

\begin{thm} \label{thm:cat and orbispace}
  The adjoint functor pair
\[\xymatrix@C=12mm{ 
  \Gamma \ : \ \bO_{\gl}\Mod  \quad \ar@<.4ex>[r] & 
  \quad \cat\ : \ \Ex^2\circ N_{\gl} \ar@<.4ex>[l]  }\]  
is a Quillen equivalence between the category
of small categories with the global model structure
and the category of orbispaces with the projective model structure.
\end{thm}

\begin{rem}[Models for the homotopy theory of orbispaces]\label{rk:other orbispc}
Informally speaking, an orbispace is the quotient of a space by the
action of a finite group, but with information about the isotropy groups of the
action built into the formalism.
When modeled by a functor from $\bO_{\gl}$ to simplicial sets,
the value at a finite group $G$ should be thought of as the $G$-fixed points
of the action of an ambient group before taking the quotient.
We briefly recall that the homotopy theory of orbispaces has various other models.
There are different formal frameworks for 
stacks and orbifolds (algebro-geometric, smooth, topological), 
and these objects can be studied with respect to various notions of `equivalence'.
The approach closest to our present context 
uses the notions of {\em topological stacks} and {\em orbi\-spaces} 
as developed by Gepner and Henriques in \cite{gepner-henriques}. 
Their framework allows for a choice of class of `allowed isotropy groups', 
and we will restrict to the case of finite groups. 
In \cite[Def.\,4.1]{gepner-henriques},
Gepner and Henriques define orbispaces as continuous contravariant functors 
from a certain topological category $\Orb$
to the category $\bT$ of compactly generated spaces
in the sense of \cite{mccord}.
The space $\Orb(K,G)$ defined in \cite[(38)]{gepner-henriques} 
is precisely the geometric realization of the simplicial set
$\bO_{\gl}(K,G)$, see \cite[Remark 4.5]{gepner-henriques}. 
Hence objectwise geometric realization provides a functor
\[ |-|\ : \ \bO_{\gl}\Mod \ \to \ 
\Fun^{\text{cts}}(\text{Orb}^{\op},\bT) \ = \ \text{Orb-spaces}\ .\]
For an orbispace $Y$, the value $|Y|(G)$ is $|Y(G)|$, and the continuous functoriality
is the composite
\[ |Y(G)|\times |\bO_{\gl}(K,G)|\ \iso \ |Y(G)\times \bO_{\gl}(K,G)|\ \xra{|Y|} \ |Y(K)|\ , \]
exploiting that geometric realization commutes with products.
Moreover, this functor and its right adjoint provide a Quillen equivalence,
see \cite[Thm.\,6.5]{schwede-shipley:equivalences}.
In combination with our Theorem \ref{thm:cat and orbispace} this shows that
the global homotopy theory of categories is Quillen equivalent to
the homotopy theory of Orb-spaces with finite isotropy groups.

In \cite[Def.\,2.1]{schwede:orbispace}, the author defined a 
different topological indexing category -- also denoted $\bO_{\gl}$ -- that is made from
spaces of linear isometric self-embeddings of $\mR^\infty$.
K{\"o}rschgen \cite[Cor.\,3.13]{koerschgen} related $\Orb$ and $\bO_{\gl}$ by a chain
of two weak equivalences of topological categories.
As a consequence, the two kinds of orbispaces indexed by 
$\Orb$ and $\bO_{\gl}$ are Quillen equivalent, see
\cite[Lemma A.6]{gepner-henriques}.
The paper \cite{schwede:orbispace} identifies orbispaces  
with the global homotopy theory of `spaces with an action
of the universal compact Lie group', 
and with the global homotopy theory of orthogonal spaces 
from \cite[Ch.\,I]{schwede:global}. 
The universal compact Lie group (which is neither compact nor a Lie group)
is a well-known object, namely 
the topological monoid $\Lc=\bL(\mR^\infty,\mR^\infty)$
of linear isometric embeddings of $\mR^\infty$ into itself.
\end{rem}

Now we turn to the analysis of the globally cofibrant small categories,
which essentially coincide with the `complexes of groups'.
In Thomason's non-equivariant model structure on $\cat$, every cofibrant
category is a poset category, compare \cite[Prop.\,5.7]{thomason:model cat cat}.
The global model structure has fewer equivalences and fibrations, 
so more categories are globally cofibrant. And indeed, the category $B G$ is
cofibrant for every finite group $G$, but obviously not a poset category
(unless $G$ is trivial).
We shall now explain that globally cofibrant categories
are in a precise way `built from posets by adding finite automorphism groups';
the rigorous statement is Theorem \ref{thm:cofibrant is complex},
saying that every globally cofibrant small category is isomorphic 
to the opposite of the Grothendieck construction of a pseudofunctor
from a poset to the 2-category $\grp$.
These pseudofunctors have been extensively studied under
the name `complexes of groups', most notably by Haefliger \cite{haefliger}
and Bridson and Haefliger \cite[Ch.\,III.$\Cc$]{bridson-haefliger}.
So in this sense, all globally cofibrant categories `are' complexes of groups.

\begin{con}[The 2-category of groups]\label{def:Grp}
We denote by $\grp$ the 2-category of groups whose objects are all groups.
For groups $G$ and $K$, the category $\grp(K,G)$
is the translation groupoid of $G$ acting by conjugation on the
set of homomorphisms. So the objects of
$\grp(K,G)$ are all group homomorphisms $\alpha:K\to G$,
and morphisms from $\alpha$ to $\beta$ are all $g\in G$ such that
$\beta=c_g\circ\alpha$, where $c_g:G\to G$ is the inner automorphism 
$c_g(\gamma) = g \gamma g^{-1}$.
Composition in the 2-category $\grp$
is composition of homomorphisms and multiplication
in the groups, i.e., given by
\[ (g:\alpha\to \beta)\circ (k:\delta\to\epsilon)\ = \ 
(g\alpha(k):\alpha\circ\delta\to \beta\circ\epsilon) \ .\]
As already mentioned in Construction \ref{con:global orbit category},
the 2-category $\grp$ is isomorphic to the full 2-subcategory of the 2-category $\cat$
with objects the classifying categories $B G$ for all groups $G$.
In more down to earth terms, the functor
\[ \grp(K,G)\ \to \ \Fun(B K, B G) \]
that sends $\alpha:K\to G$ to $B\alpha:B K\to B G$
and $g\in G$ to the natural isomorphism $B\alpha\Longrightarrow B(c_g\circ\alpha)$
given by $g$ 
is an isomorphism of categories,
and these isomorphisms are compatible with the composition functors
for varying groups.
\end{con}

Pseudofunctors (also called {\em weak functors}) 
are a notion of morphism between 2-categories that only commute with composition 
up to specified coherent invertible 2-cells.
The concept goes back to Grothendieck, and the general definition can 
for example be found in \cite[Def.\,7.5.1]{borceux}.
We will only consider pseudofunctors from poset categories to the 2-category $\grp$,
and we'll make the definition explicit below. Since all 2-cells in $\grp$
are invertible, pseudofunctors to $\grp$ are the same as lax functors
and op-lax functors.

\begin{defn}\label{def:complex of groups} 
A {\em complex of groups} is a pseudofunctor
\[ \Gc \ : \ P \ \to \ \grp  \]
to the 2-category of groups, where $P$ is a poset.
\end{defn}

We expand the definition.
A pseudofunctor $\Gc:P\to\grp$ consists of the following data:
\begin{itemize}
\item a group $\Gc(x)$ for every object $x$ of $P$,
\item a group homomorphism $\Gc(x,y):\Gc(x)\to \Gc(y)$ for every comparable pair
  $x\leq y$ in $P$, and
\item a group element $\Gc(x,y,z)\in \Gc(z)$ for every comparable triple
$x\leq y\leq z$ in $P$.
\end{itemize}
Moreover, this data has to satisfy:
\begin{itemize}
\item for every element $x$ of $P$ the homomorphism $\Gc(x,x)$
  is the identity of $\Gc(x)$,
  \item for every pair of comparable elements $x\leq z$ of $P$, we have $\Gc(x,x,z)=\Gc(x,z,z)=1$,
\item for every triple of comparable elements $x\leq y\leq z$, the relation
  \begin{equation} \label{eq:lax_functoriality}
 c_{\Gc(x,y,z)}\circ \Gc(x,z) \ = \ \Gc(y,z)\circ \Gc(x,y)    
  \end{equation}
  holds as homomorphisms $\Gc(x)\to \Gc(z)$, where $c_{\Gc(x,y,z)}$ is conjugation
  by $\Gc(x,y,z)$, and
\item for every quadruple of comparable elements $w\leq x\leq y\leq z$,
  the relation
  \begin{equation}\label{eq:cocycle}
  \Gc(x,y,z)\cdot\Gc(w,x,z)\ = \ \Gc(y,z)(\Gc(w,x,y))\cdot \Gc(w,y,z)
  \end{equation}
  holds in the group $\Gc(z)$.
\end{itemize}
The group $\Gc(x)$ is sometimes called the `local group at $x$',
$\Gc(x,y)$ is the `transition homomorphism',
and $\Gc(x,y,z)$ is the `twisting element'.
The first two conditions encode that the transition homomorphisms
and twisting elements are strictly unital.
Relation \eqref{eq:lax_functoriality} says that
the transition homomorphisms are only functorial in a lax sense,
namely up to the 2-morphisms in $\grp$ specified by the twisting elements.
The relation \eqref{eq:cocycle} is a coherence relation,
saying that the two 2-morphisms
from $\Gc(y,z)\circ\Gc(x,y)\circ\Gc(w,x)$ to $\Gc(w,z)$
arising from the two ways of parenthesizing the triple composite are equal.

Complexes of groups generalize the earlier notion of {\em graphs of groups}
that arose in geometric group theory in the work of Bass \cite{bass}
and Serre \cite{serre:trees}.
Our Definition \ref{def:complex of groups} is a slight variation of
the notion of complex of groups used by Haefliger \cite[Def.\,2.1]{haefliger}
and Bridson-Haefliger \cite{bridson-haefliger}.
There are two differences: we index by posets, whereas
Haefliger \cite{haefliger} restricts to ordered simplicial complexes,
and Bridson-Haefliger \cite{bridson-haefliger} use the 
more general `scwols' (`small categories without loops').
More importantly, we allow arbitrary group homomorphisms for
the transition homomorphisms $\Gc(x,y):\Gc(x)\to\Gc(y)$,
whereas most other sources insist on {\em injective} homomorphisms. 
The motivation for the injectivity condition is that
in all complexes of groups that arise as global quotients,
the transition homomorphisms are injective.

\begin{eg}
Every strict functor to the 1-category
of groups and homomorphisms can be considered as a pseudofunctor to the 2-category $\grp$
in which all twisting elements are identities.
Such complexes of groups are usually called {\em simple}.
\end{eg}

\begin{con}
Let $\Gc:P\to \grp$ be a complex of groups in the sense 
of Definition \ref{def:complex of groups}.  
Postcomposition with the strict 2-functor $B:\grp\to\cat$ produces a
a pseudofunctor
\[ P \ \xra{\ \Gc \ }\ \grp \ \xra{\ B \ }\ \cat \]
to the 2-category of small categories. 
Every pseudofunctor $F:P\to \cat$
has an associated {\em Grothendieck construction} $P\smallint F$, 
compare \cite[Def.\,3.1.2]{thomason:hocolim};
if the pseudofunctor happens to be a strict functor, then
this construction specializes to the one discussed in Example \ref{eg:Grothendieck construction}.

In the special case of a pseudofunctor arising from a complex of groups, 
and when the index category $P$ is a poset or a scwol, 
this Grothendieck construction is simply called
the {\em category associated to the complex of groups};
we write $c(\Gc)$ for this associated category, which 
has the following explicit description, compare
\cite[2.4]{haefliger} or \cite[Ch.\,III.$\Cc$, 2.8]{bridson-haefliger}.
The objects of $c(\Gc)$ are the elements of $P$, and morphisms are given by
\[ c(\Gc)(x,y)\ = \
\begin{cases}
  \Gc(y)  & \text{ if $x\leq y$, and}\\
  \ \emptyset & \text{ if $x\not\leq y$.}
\end{cases}
\]
For comparable elements $x\leq y\leq z$ of $P$, composition 
\[ \circ \ : \ c(\Gc)(y,z)\times c(\Gc)(x,y)\ \to \ c(\Gc)(x,z) \]
is defined by 
\begin{equation}\label{eq:define_composition}
  g\circ h \ = \ g\cdot \Gc(y,z)(h)\cdot \Gc(x,y,z)\  ,  
\end{equation}
where the right hand side is the product in the group $\Gc(z)$.
The first two conditions in Definition \ref{def:complex of groups} 
ensure that the neutral element $1\in\Gc(y)$ in the group structure
is an identity of $y$ in the category $c(\Gc)$.
Relations \eqref{eq:lax_functoriality} and \eqref{eq:cocycle}
enter into the verification that the composition \eqref{eq:define_composition} is associative.
\end{con}

    The next proposition is well known and can be found, in slightly different forms,
    in \cite[p.\,283]{haefliger} and \cite[III.$\Cc$, Prop.\,A.6]{bridson-haefliger}.
    Besides the different kinds of indexing categories for the complexes
    of groups, the main difference is that we allow arbitrary 
    transition homomorphisms, whereas \cite{bridson-haefliger, haefliger}
    require these to be injective.
    To characterize categories associated to the more restricted
    kind of complexes of groups, one must add the condition
    that for all objects $x$ and $y$, 
    the action of $\Cc(x,x)$ on $\Cc(x,y)$ by precomposition is free.

  \begin{prop}\label{prop:characterize complexes} 
    A small category $\Cc$ is isomorphic to the category associated to
    a complex of groups in the sense of Definition \ref{def:complex of groups} 
    if and only if it satisfies the following two conditions:
      \begin{enumerate}[\em (a)]
      \item If $x,y$ are two objects of $\Cc$ that admit morphisms $x\to y$
        and $y\to x$, then $x=y$.
      \item For every pair of $\Cc$-morphisms $f,f':x\to y$ with the same source
        and the same target, there is a unique morphism $\alpha:y\to y$ 
        such that $f'=\alpha\circ f$.
    \end{enumerate}
  \end{prop}
  \begin{proof}
    The verification that the category
    associated with a complex of groups indexed by a poset
    has properties (a) and (b) is straightforward, and we omit it.
    The more interesting part is the other implication;
    since this result is standard, we will be brief and just indicate how to obtain 
    the complex of groups from the category $\Cc$.

    Property (a) says that the object set of $\Cc$ becomes a poset by declaring
    $x\leq y$ if and only if there exists a morphism $x\to y$; 
    we write $\pos(\Cc)$ for this poset.
    Property (b) can be phrased as two separate conditions, namely:
    \begin{itemize}
    \item Every endomorphism in $\Cc$ is an automorphism, and 
    \item for all objects $x,y$ of $\Cc$ with $x\leq y$, 
      the action of the group $\Cc(y,y)$ 
      on $\Cc(x,y)$ by postcomposition is free and transitive.
    \end{itemize}
    We construct a complex of groups
    \begin{equation}\label{eq:aut}
      \aut\ : \ \pos(\Cc)\ \to \ \grp\ .       
    \end{equation}
    The construction depends on a choice of $\Cc$-morphism $f_{y,x}:x\to y$ 
    for every pair of comparable objects in $\pos(\Cc)$.
    We insist that $f_{x,x}$ is the identity of $x$ for every object of $\Cc$.

    The pseudofunctor $\aut$ sends an object $x$ of $\pos(\Cc)$ to 
    $\aut(x) = \Cc(x,x)$, the automorphism group of $x$ in $\Cc$. 
    For every pair $x\leq y$ of comparable objects, a map
    \[ \aut(x,y)\ : \ \Cc(x,x)\ \to \ \Cc(y,y)\]
    is defined by requiring that
    \begin{equation}\label{eq:define_aut_2}
     \aut(x,y)(\beta)\circ f_{y,x} \ = \ f_{y,x}\circ \beta   
    \end{equation}
    for $\beta\in \Cc(x,x)$. 
    If $\gamma\in\Cc(x,x)$ is another element, then
    \begin{align*}
      \aut(x,y)(\beta)\circ \aut(x,y)(\gamma)\circ f_{y,x}\ &=\ 
      \aut(x,y)(\beta)\circ f_{y,x}\circ \gamma \\ 
      &= \ f_{y,x}\circ \beta \circ \gamma   \     
      = \ \aut(x,y)(\beta \circ \gamma) \circ f_{y,x}\ .     
    \end{align*}
    So the uniqueness clause in property (b) shows that $\aut(x,y)$ is 
    a group homomorphism. The normalization condition $f_{x,x}=\Id_x$
    ensures that $\aut(x,x)$ is the identity.

    For every triple of comparable $\Cc$-objects $x\leq y\leq z$,
    there is a unique element $\aut(x,y,z)\in C(z,z)$ such that 
    \begin{equation}\label{eq:define_aut_3}
      f_{z,y}\circ f_{y,x}\ = \  \aut(x,y,z) \circ f_{z,x}\ . 
    \end{equation}
    The normalization conditions $f_{x,x}=\Id_x$ and $f_{z,z}=\Id_z$
    ensure that $\aut(x,x,z)=\aut(x,z,z)=1$.
    Using \eqref{eq:define_aut_2} and \eqref{eq:define_aut_3} repeatedly gives
    \[ \aut(x,y,z) \circ \aut(x,z)(\beta)\circ f_{z,x} \ 
    = \   \aut(y,z)(\aut(x,y)(\beta))  \circ   \aut(x,y,z) \circ f_{z,x} \]
for all $\beta\in\Cc(x,x)$. Uniqueness in (b) implies that 
\[  \aut(x,y,z) \circ \aut(x,z)(\beta)\ =\
\aut(y,z)(\aut(x,y)(\beta)) \circ \aut(x,y,z) \ ,\]
or, equivalently, condition \eqref{eq:lax_functoriality} for $\aut:\pos(\Cc)\to\grp$.
We omit the verification of the cocycle condition
\eqref{eq:cocycle}, which is similarly straightforward.
This completes the construction of the complex of groups \eqref{eq:aut}.

An isomorphism of categories $\kappa: c(\aut) \to\Cc$
is the identity on objects, and given on morphisms by
\[ \kappa \ : \ c(\aut)(x,y) \ = \ \aut(y)\ \to \ \Cc(x,y)\ , \quad
\kappa(\gamma)\ = \ \gamma\circ f_{y,x}\ ,\]
where $x\leq y$ are comparable elements in $\pos(\Cc)$.
Functoriality is essentially built into the definition of $\aut$:
\begin{align*}
\kappa(\delta)\circ \kappa(\gamma)
\ &= \
\left( \delta\circ f_{z,y}\right)\circ   \left( \gamma\circ  f_{y,x}\right)
\\ 
_\eqref{eq:define_aut_2} \ &= \
\ \delta\circ \aut(y,z)(\gamma)\circ f_{z,y}\circ  f_{y,x}\\
_\eqref{eq:define_aut_3}\ &= \
\ \delta\circ \aut(y,z)(\gamma)\circ \aut(x,y,z)\circ f_{z,x}\\ 
&= \ 
\kappa( \delta\circ \aut(y,z)(\gamma)\circ\aut(x,y,z) )\
=_\eqref{eq:define_composition} \ \kappa( \delta\circ\gamma)
\end{align*}
where $\delta\in\aut(z)$ and $\gamma\in\aut(y)$.
Because $\aut(y)$ acts freely and transitively on $\Cc(x,y)$,
the assignment is bijective on morphism sets, and hence an isomorphism of categories.
\end{proof}

\begin{rem}
  Every poset can be considered as a category with a unique morphism between
  any comparable pair of elements. This construction extends to
  a fully faithful functor from the category of posets 
  and weakly monotone maps to the category of small categories. 
  This functor has a left adjoint
  \[ \pos \ :  \ \cat \ \to \ \poset \]
  that is also left inverse, as follows. 
  We define an equivalence relation on the object
  set of a small category by declaring $x\sim y$ if there exists a morphism $x\to y$
  and a morphism $y\to x$. Then $\pos(\Cc)$ is the set of equivalence
  classes of objects of $\Cc$, ordered by declaring $[x]\leq [y]$
  if and only if there exists a morphism $x\to y$.
    
  If $\Cc$ satisfies conditions (a) and (b)
  of Proposition \ref{prop:characterize complexes}, then the proof shows that
  the indexing poset for the complex of groups can be taken to be $\pos(\Cc)$.
  This justifies our notation `$\pos(\Cc)$' in that proof.
\end{rem}

Now we come to the identification of globally cofibrant categories
as complexes of groups.

\begin{thm}\label{thm:cofibrant is complex}
  Let $\Cc$ be a small category that is cofibrant in the global model structure
  of Theorem \ref{thm:global cat}. Then the opposite category $\Cc^{\op}$
  satisfies properties {\em (a)}
  and {\em (b)} of Proposition \ref{prop:characterize complexes}, 
  and all automorphism groups in $\Cc$ are finite.
  Hence every cofibrant category in the global model structure on $\cat$ is isomorphic
  to the opposite of the category associated to a complex of finite groups indexed by a poset category.
\end{thm}
\begin{proof}
  We let $\Xc$ be the class of small categories $\Cc$ with finite automorphism groups
  that satisfy property (a) of Proposition \ref{prop:characterize complexes}
  and the following property:
\begin{itemize}
\item[(b)$^{\op}$] 
        For every pair of $\Cc$-morphisms $f,f':x\to y$ with the same source
        and the same target, there is a unique morphism $\omega:x\to x$ 
        such that $f'=f\circ\omega$.
\end{itemize}
Property (b)$^{\op}$ is equivalent to property (b) of Proposition \ref{prop:characterize complexes}
for the opposite category $\Cc^{\op}$.
Moreover, property (a) holds in $\Cc$ if and only if if holds in $\Cc^{\op}$.
So the theorem amounts to showing that every cofibrant small category
belongs to the class $\Xc$.

  We recall from \eqref{eq:I_for_F-proj_on_GT} the functors
  \[ 
  i_n\times  B G  \ : \ c(\Sd^2(\partial\Delta[n]))\times  B G \
  \to \ c(\Sd^2(\Delta[n])) \times B G\ . \]
  For $n\geq 0$ and $G$ in a set of representatives of the isomorphism classes
  of finite groups, these functors form the set $I_{\gl}$ of generating
  cofibrations for the global model structure of Theorem \ref{thm:global cat}. 
  We consider a pushout square of categories:
  \[ \xymatrix{\coprod_{i\in I}     c(\Sd^2(\partial\Delta[n_i]))\times B G_i 
    \ar[r] \ar[d] & 
    \coprod_{i\in I} c(\Sd^2(\Delta[n_i])) \times B G_i  \ar[d] \\
    C \ar[r] & D
  } \]
  The upper horizontal functor is a Dwyer map, and so
  the pushout $D$ has the explicit description provided 
  in Construction \ref{con:Dwyer pushout}.
  We claim that whenever $C$ belongs to the class $\Xc$, then so does $D$.

  As in Construction \ref{con:Dwyer pushout}
  we let $V$ be the full subcategory of $B=\coprod_{i\in I} c(\Sd^2(\Delta[n_i])) \times B G_i$
  whose objects are the ones 
  that do not belong to $A=\coprod_{i\in I} c(\Sd^2(\partial\Delta[n_i])) \times B G_i$.
  Construction \ref{con:Dwyer pushout} shows that $D$ contains $C$ and $V$
  as disjoint full subcategories containing all objects;
  moreover, there are no morphisms from objects in $V$ to objects in $C$.
  Since the subcategories $C$ and $V$ satisfy condition (a), and every object of $D$
  belongs to $C$ or $V$, the category $D$ again satisfies condition (a).
  
  To show that the pushout category $D$ has property (b)$^{\op}$
  we consider two objects $x,y$ of $D$.
  First of all, the monoid  $D(x,x)$ is a group:
  $x$ belongs to $C$, this is true by hypothesis because $C$ is
  a full subcategory of $D$.
  If $x$ belongs to $V$, this is true by inspection of the category $V$,
  which is also a full subcategory of $D$.
  Secondly, we claim that the action of the group $D(x,x)$ on the set $D(x,y)$
  by precomposition is free and transitive.
  If $x$ and $y$ both belong to $C$, this is true by hypothesis because $C$ is
  a full subcategory of $D$.
  If $x$ and $y$ both belong to $V$, this is true by inspection of the category $V$,
  which is also a full subcategory of $D$.
  There are no $D$-morphisms from objects in $V$ to objects in $C$,
  so the only other case to consider is when $x\in C$ and $y\in V$.
  The morphism set $D(x,y)$ is empty if $y$ does not belong to the cosieve
  generated by $A$ in $B$, so there is nothing to show in this case.
  If $y$ belongs to the cosieve generated by $A$ in $B$, then
  \[ D(x,x)\ = \ C(x,x)\text{\quad and\quad} D(x,y)\ = \ C(x,\bar y) \]
  for a specific $C$-object $\bar y$; moreover, the precomposition action of
  $D(x,x)$ on $D(x,y)$ is the precomposition action of
  $C(x,x)$ on $C(x,\bar y)$. The latter is free and transitive because $C$ belongs
  to the class $\Xc$. Hence we have shown that the category $D$ again belongs to the class $\Xc$.
   
  Now we consider fully faithful functors $f_n:Y_n\to Y_{n+1}$ between small categories
  such that $Y_n$ belongs to the class $\Xc$ for every $n\geq 0$.
  We let $Y_\infty$ be a colimit of the sequence. Then the canonical functor
  $Y_n\to Y_\infty$ is fully faithful for every $n\geq 0$, and every pair of
  objects of $Y_\infty$ arises from $Y_n$ for some finite $n$.
  So $Y_\infty$ again belongs to the class $\Xc$.
  Since the empty category belongs to $\Xc$, this shows
  that every $I_{\gl}$-cell complex belongs to the class $\Xc$.
  
  The small object argument shows that every globally cofibrant category
  is a retract of an  $I_{\gl}$-cell complex, compare the proof of 
  Theorem \ref{thm:M cat}. 
  The class $\Xc$ is closed under retracts, so this concludes the proof.
\end{proof}

We close by relating our global homotopy theory of small categories
to the homotopy theory of $G$-categories for a fixed group $G$.

\begin{eg}[Induced categories]
We let $G$ be a finite group. A {\em $G$-category}
is a small category equipped with a left action of $G$ by automorphisms of
categories. We denote by $G\cat$ the category of small $G$-categories and
$G$-equivariant functors.

We let $E G$ be the translation groupoid of $G$: its
object set is $G$ and its morphism set is $G\times G$, where $(g,h)$
has source $h$ and target $g$. Composition is given by $(g,h)\circ(h,k)=(g,k)$.
The group $G$ acts on the category $E G$  by right translation on objects and
morphisms. We obtain an adjoint functor pair
\[\xymatrix@C=10mm{ 
E G\times_G - \ : \ G\cat  \quad \ar@<.4ex>[r]^-{} & 
\quad \cat\quad : \ \Fun(E G,-)\ar@<.4ex>[l]^-{} }\]  
between the category of small $G$-categories and the category
of small categories. As we shall now explain, this adjoint functor
pair is the categorical incarnation of the `global quotient orbispace'
(for the left adjoint $E G\times_G -$)
and of the `underlying $G$-space' (for the right adjoint $\Fun(E G,-)$),
respectively.

We observe that the right adjoint functor $\Fun(E G,-)$
takes global equivalences to $G$-equivariant functors
that induces weak equivalence on $H$-fixed categories
for all subgroups $H$ of $G$; these are the 
weak equivalences of $G$-categories considered in \cite{BMOOPY}.
Indeed, precomposition with the quotient functor $E G\to (E G)/H$ 
and the equivalence
\[ B H \ \to \ (E G)/H \ , \quad \ast \ \longmapsto e H\ , \quad
h \longmapsto \ (h, e) H \]
induce equivalences of categories
\begin{equation}\label{eq:H-fix EG}
\Fun(E G,\Cc)^H \ \iso \   \Fun( (E G)/H,\Cc) \ \simeq \
\Fun( B H,\Cc) \ \iso \  H\Cc \ .
\end{equation}
Equivalences of categories become homotopy equivalences on nerves,
so we conclude that the nerves of the categories $H \Cc$ and
$\Fun(E G,\Cc)^H$ are naturally homotopy equivalent.
In particular, the right adjoint $\Fun(E G,-)$ descends to a
functor on homotopy categories
\[ \Ho(\Fun(E G,-)) \ : \ \Ho^{\gl}(\cat)\ \to \ \Ho(G\cat) \ .\]  
The adjoint pair $(E G\times_G -,\Fun(E G,-))$ is {\em not} a Quillen functor
pair for the global model structure of Theorem \ref{thm:global cat}.
For example, the discrete $G$-category $G/H$ is cofibrant
in the model structure of \cite{BMOOPY}, for every subgroup $H$ of $G$;
however, the category $E G\times_G (G/H)\iso (E G)/H$ is not globally
cofibrant unless $H=G$. This phenomenon is not homotopically
significant and easy to fix: we let $\Gc$ be the class of finite connected
groupoids. Every finite connected groupoid is equivalent to $B G$ for
some finite group $G$, so Corollary \ref{cor:equivalent}
shows that the $\Gc$-model structure on the category of small categories
is Quillen equivalent to the global model structure.
Since $(E G)/H$ is a finite connected groupoid, the natural isomorphism of categories
$\Fun(E G,\Cc)^H \iso \Fun( (E G)/H,\Cc)$ shows that 
$(E G\times_G -,\Fun(E G,-))$ is a Quillen functor pair with respect
to the model structure on $G\cat$ established in \cite{BMOOPY}
and the $\Gc$-model structure on $\cat$.

The following diagram of homotopy categories summarizes
the homotopical content of the adjoint functor pair $(E G\times_G -,\Fun(E G,-))$:
\[ \xymatrix@C=30mm@R=7mm{ 
\Ho(G\cat)\ar[dd]_{\Ho(N)}^\iso\ar@<.4ex>[r]^-{L(E G\times_G -)}&  
\Ho^{\gl}(\cat)\ar@<.4ex>[l]^-{\Ho(\Fun(E G,-))}
 \ar[d]^{\Ho(N_{\gl})}_\iso \\
& \Ho(\bO_{\gl}\Mod)\ar[d]^\iso  \\
\Ho(G\bT)\ar@<.4ex>[r]^-{L(\bL_{G,V})}
  &  \Ho^{\gl}(\spc)\ar@<.4ex>[l]^-{\Ho(Y \longmapsto Y(\Uc_G))}
} \]
In the lower part of the diagram, $G\bT$ denotes the category of $G$-spaces,
and $\Ho(G\bT)$ its homotopy category with respect to $G$-weak equivalences.
Moreover, $\spc$ denotes the category of orthogonal spaces
in the sense of \cite[Def.\,1.1.1]{schwede:global},
and $\Ho^{\gl}(\spc)$ is its homotopy category with respect to the global equivalences of
\cite[Def.\,1.1.2]{schwede:global}.
Orthogonal spaces under global equivalences are a model for unstable global homotopy
theory, and they support a model category structure that is Quillen equivalent
to orbispaces by the results of \cite{schwede:orbispace}.

The horizontal arrows in the diagram are pairs of adjoint functors, 
and all vertical functors are equivalences of categories.
The left vertical nerve functor is an equivalence
of homotopy categories by \cite{BMOOPY}.
The lower right vertical equivalence was discussed in
Remark \ref{rk:other orbispc}.
The lower adjoint functor pair takes 
a $G$-space to the `induced orthogonal space', 
and an orthogonal space to the `underlying $G$-space';
we refer to \cite[Rk.\,1.2.24]{schwede:global} for a more detailed discussion.
We omit the verification that the diagram commutes up to natural isomorphism.
The equivalence \eqref{eq:H-fix EG} may serve as evidence: it says that turning
a category $\Cc$ into a $G$-space by the two ways around the diagram leads to naturally
weakly equivalent $H$-fixed points for all subgroups $H$ of $G$.
\end{eg}

\end{document}